\documentclass[11pt]{article}
\title{The satisfiability threshold and solution space of random uniquely extendable constraint satisfaction problems}
\author{Pu Gao\thanks{Research supported by NSERC RGPIN-04173-2019.} \\ University of Waterloo \\ pu.gao@uwaterloo.ca \and Theodore Morrison\thanks{Research supported by an Ontario Graduate Scholarship} \\ University of Waterloo\\ tmorriso@uwaterloo.ca}
\date{}
\usepackage{amsfonts, amsmath, amssymb, amsthm, bigints, bbm,tikz,scrextend}
\usepackage[shortlabels]{enumitem}
\usetikzlibrary{graphs}
\usepackage[dvipsnames]{xcolor}

\newtheorem{theorem}{Theorem}[section]

\newtheorem{lemma}[theorem]{Lemma}

\newtheorem{definition}[theorem]{Definition}

\newtheorem{conjecture}[theorem]{Conjecture}

\newtheorem{claim}[theorem]{Claim}
\newtheorem{remark}[theorem]{Remark}
\newtheorem{problem}[theorem]{Problem}

\newcommand{\ZZ}{\mathbb{Z}}
\newcommand{\FF}{\mathbb{F}}
\newcommand{\RR}{\mathbb{R}}

\newcommand{\eps}{\varepsilon}
\newcommand{\fhi}{\varphi}
\newcommand{\lt}{\left}
\newcommand{\rt}{\right}

\renewcommand{\subset}{\subseteq}

\newcommand{\Id}{\operatorname{Id}}

\newcommand{\supp}{\operatorname{supp}}

\theoremstyle{definition}

\DeclareSymbolFont{symbolsC}{U}{txsyc}{m}{n}
\DeclareMathSymbol{\strictif}{\mathrel}{symbolsC}{74}
\DeclareMathSymbol{\strictfi}{\mathrel}{symbolsC}{75}
\DeclareMathSymbol{\strictiff}{\mathrel}{symbolsC}{76}
\usepackage{fullpage}

\def\pr{{\mathbb P}}
\newcommand\remove[1]{{}}
\newcommand\ind[1]{\boldsymbol{1}_{\{#1\}}}

\begin{document}

\maketitle

\begin{abstract}
    We study the satisfiability threshold and solution-space geometry of random constraint satisfaction problems defined over uniquely extendable (UE) constraints. Motivated by a conjecture of Connamacher and Molloy, we consider random $k$-ary UE-SAT instances in which each constraint function is drawn, according to a certain distribution $\pi$, from a specified subset of uniquely extendable constraints over an $r$-spin set. We introduce a flexible model $H_n(\pi,k,m)$ that allows arbitrary distributions $\pi$ on constraint types, encompassing both random linear systems and previously studied UE-SAT models. Our main result determines the satisfiability threshold for a wide family of distributions $\pi$. Under natural reducibility or symmetry conditions on $\supp(\pi)$, we prove that the satisfiability threshold of $H_n(\pi,k,m)$ coincides with the classical $k$-XORSAT threshold.

\end{abstract}

\section{Introduction}

Random constraint satisfaction problems (CSPs) were introduced to study how computationally hard problems--especially NP-complete problems--behave on average, rather than in the worst case. CSPs are central objects of study across multiple disciplines: in combinatorics (graph and hypergraph colouring, independent sets, matchings, regular subgraphs), in computer science ($k$-SAT, $k$-XORSAT), and in physics (the hard-core model, the monomer–dimer model, the Ising model).

A typical CSP instance consists of a set of $n$ variables and a set of $m$ constraints. The variables take values from a predefined set $\Omega$, often called the set of spins in physics. Each constraint is specified by a set called the {\em constraint variables}, and a function $\psi$ called the {\em constraint function}. The constraint variables is a small subset of $k$ variables $x_1,\ldots, x_k$ (for some $k\ge 1$) from the total set of $n$ variables, and the constraint function $\psi: \Omega^k\to \{0,1\}$ evaluates any $k$-tuple in $\sigma\in\Omega_k$ to 1 (we say the constraint is satisfied by $\sigma$) or 0 (we say the constraint is not satisfied by $\sigma$). A constraint is called $k$-ary if its constraint function is defined on $k$ variables, i.e.\ $\Omega^k$. %For a constraint function $\psi$, we call $\psi^{-1}(1)$ the truth values of a constraint whose constraint function is $\psi$. 
Different constraint functions arise from different types of constraint satisfaction problems.  For example, in graph colouring, each edge induces a binary constraint ($k=2$): its endpoints $x,y$ satisfy the constraint if $x$ and $y$ receive distinct values. In some CSPs  all constraints have the same constraint function (e.g.\ graph colouring), while in others they may vary. For instance, in a system of linear equations over a field ${\mathbb F}$, the constraint function corresponding to an equation $a_1x_1+\ldots+a_kx_k=b$ is determined by both the coefficients $(a_1,\ldots, a_k)$ and the right hand side $b$, and the equations within a system may each have different coefficients and right-hand sides.

In a random $k$-ary CSP, each constraint first selects its constraint function from a specified set $\Gamma$ according to a probability distribution $\pi$, and then chooses $k$ (ordered) variables it involves uniformly at random from the set of $n$ variables. This is the simplest random CSP model that has been extensively studied across many problems~\cite{duboismandler2002XORSAT,connamacher2012satisfiability,ayre2020satisfiability,coja2020replica,achlioptas2006twomomentsNAESAT}, though more complicated random models have also been investigated and analysed~\cite{coja2020rank,coja2024full,lelarge2013bypassing}.

Research on random CSPs focuses heavily on phase transition phenomena, which arise ubiquitously across a wide range of problems and models~\cite{duboismandler2002XORSAT,connamacher2012satisfiability,ayre2020satisfiability,coja2020replica,achlioptas2006twomomentsNAESAT}. Roughly speaking, as the number of constraints $m$ increases, the solution space of a random CSP undergoes a sequence of structural transitions: first, solutions begin to partition into distinct clusters, marking an algorithmic barrier for ``local'' algorithms; next, solutions within each cluster start to agree on a large fraction of variables; then, long-range correlations emerge, so that the joint marginal distribution of two variables is no longer close to the product of their individual marginals in a uniformly random solution; and finally, the entire system becomes unsatisfiable. These transitions--commonly referred to as the clustering, freezing, replica-symmetry-breaking (RSB), and satisfiability thresholds--form the core of modern studies in random CSPs.

As noted above, research on random CSPs is vast~\cite{achlioptas2006twomomentsNAESAT,ayre2020satisfiability,coja2020replica,connamacher2012satisfiability,dingslysun2022kSAT,duboismandler2002XORSAT,goerdt2012beyondxorsat,ibrahimi2015solutionclustering,mezard2003spinglass}. The most relevant strands for this work concern random linear equations and random uniquely extendable CSPs. The study of random linear equations dates back to the determination of the satisfiability threshold of random 3-XORSAT~\cite{duboismandler2002XORSAT}, which was later extended to $k$-XORSAT~\cite{dietzfelbinger2010XORSATthreshold,pittelsorkin2016XORSATthreshold} and to linear equations over $GF(3)$ and $GF(4)$~\cite{goerdt2012beyondxorsat}. Finally, the satisfiability threshold for linear equations over any finite field is determined by Ayre et al.~\cite{ayre2020satisfiability}. Another extension from the study of random $k$-XORSAT is a model of uniquely extendable satisfiability problem (UE-SAT) introduced by Connamacher and Molloy~\cite{connamacher2012satisfiability}. The motivation for introducing UE-SAT, as noted in~\cite{connamacher2012satisfiability},  stems from the observation that the proof in~\cite{duboismandler2002XORSAT} relies primarily on the property of the {\em unique extendability} of a linear equation, formally defined in Definition~\ref{def:UE} below. Connamacher and Molloy~\cite{connamacher2012satisfiability} conjectured the satisfiability threshold and proved a special case of it. The main purpose of this paper is to address that conjecture. Before doing so, we introduce the necessary definitions.

\begin{definition}[Uniquely Extendability]\label{def:UE}
Let $k\ge 1$ be an integer. A $k$-ary constraint function $\psi$ is called uniquely extendable, if for every $1\le i\le k$ and for every $\sigma_1,\ldots,\sigma_{i-1},\sigma_{i+1},\ldots,\sigma_k\in\Omega$, there is a unique $\tau\in \Omega$ such that $\psi(\sigma_1,\ldots,\sigma_{i-1},\tau,\sigma_{i+1},\ldots,\sigma_k)=1$.  
\end{definition}

Fixing integers $k$ and $r$, the set $\Lambda_{k,r}$ of all $k$-ary uniquely extendable constraint functions over a set of $r$ spins (i.e.\ $|\Omega|=r$) is finite. Connamacher and Molloy's model on random UE-SAT is given as below. Let $V_n=\{v_1,\ldots, v_n\}$ be a set of $n$ variables. Independently draw $m$ uniquely extendable constraints $\psi_1,\ldots,\psi_m$ such that for each $1\le j\le m$, the constraint function of $\psi_j$ is chosen uniformly from $\Lambda_{k,r}$; and the $k$ constraint variables of $\psi_j$ are chosen uniformly at random from $V_n$, and are chosen independently of the constraint function of $\psi_j$. Connamacher and Molloy called this the $(k,r)$-UE-SAT.

Before stating Connamacher and Molloy's conjecture on the satisfiability threshold of the $(k,r)$-UE-SAT, it is necessary to call the attention to the satisfiability threshold to random $k$-XORSAT. Given a sequence of probability spaces indexed by $n$, we say a sequence of events $A_n$ occurs asymptotically almost surely (a.a.s.), if $\lim_{n\to\infty} \pr[A_n]=1$.

\begin{theorem}[\cite{pittel2016satisfiability}]\label{thm:equation-sat} Let $k\ge 3$ be a fixed integer, and let 
   \begin{align*}
            \rho_{k,d} &= \sup\lt\{x\in [0,1]: x=1-\exp(-dx^{k-1})\rt\}\\
            d_k&=\inf\lt\{d>0:\rho_{d,k}-d\rho_{d,k}^{k-1}+(1-1/k)d\rho_{d,k}^k<0\rt\}.
        \end{align*}
Suppose that $\Phi$ is a random $k$-XORSAT instance with $n$ variables and $m$ $k$-XORSAT clauses. Then, for every $\eps>0$, a.a.s.\ $\Phi$ is satisfiable if $m<(d_k/k-\eps)n$ and a.a.s.\ $\Phi$ is unsatisfiable if $m>(d_k/k+\eps)n$.

\end{theorem}

\begin{conjecture}[Connamacher and Molloy~\cite{connamacher2012satisfiability}] \label{conj:Molloy} Let $(k,r)\ge (3,2)$ be fixed. Let $d_k$ be given as in Theorem~\ref{thm:equation-sat}.  Suppose that $\Phi$ is a random $(k,r)$-UE-SAT instance with $n$ variables and $m$ $k$-ary uniquely extendable constraints. Then, for every $\eps>0$, a.a.s.\ $\Phi$ is satisfiable if $m<(d_k/k-\eps)n$ and a.a.s.\ $\Phi$ is unsatisfiable if $m>(d_k/k+\eps)n$. 
\end{conjecture}

Note that Conjecture~\ref{conj:Molloy} asserts that the satisfiability threshold for a random $(k,r)$-UE-SAT is independent of $r$. This is not surprising, as an analogous phenomenon occurs for random linear equations over finite fields, where the satisfiability threshold is always $d_k/k$, regardless of the field order~\cite{ayre2020satisfiability}. Connamacher and Molloy confirmed that Conjecture~\ref{conj:Molloy} holds for $(k,r)=(3,4)$. These are the smallest values for which the $(k,r)$-UE-SAT problem is non-trivial, as $(2,r)$-UE-SAT does not have a sharp satisfiability threshold, and it was pointed out~\cite{Connamacher2008ThresholdPI} that $(k,r)$-UE-SAT can be reduced to a system of linear equations over $\FF_r$ for $r=2,3$.

There are many $k$-ary UE-SAT problems whose constraint types are  members of $\Lambda_{k,r}$, and these can be of very different nature. Some representative examples include:

\begin{itemize}
\item Solving a system of equations over a group $G_r$ of order $r$, where each equation is of form $\prod_{i=1}^k \sigma_i=1$.
\item Solving a system of linear equations over $GF(r)$ if $r$ is a prime power;
%\item $\psi(\sigma_1,\ldots,\sigma_k)=1$ if $\prod_{i=1}^k \sigma_i=\sigma$ where $\Omega={\mathbb S}_r$ the permutation group on $r$ elements, and $\sigma\in \Omega$;
\item Orient every edge of a $k$-regular graph so that the out-degree of every vertex is even.
\item Colour the vertices of a $k$-uniform hypergraphs with $k$ colours so that the set of vertices incident to a common hyperedge receive distinct colours.
\end{itemize}

A random system where constraint types are chosen uniformly at random from $\Lambda_{k,r}$ is somewhat artificial; for instance, one would not naturally mix equations over different groups, linear equations over fields, and hypergraph colouring constraints in the same system. Instead, we propose a more general and often more natural model in which the distribution on $\Lambda_{k,r}$ is not necessrily uniform. In particular, our model allows the use of atomic distributions on $\Lambda_{k,r}$, corresponding to studying random systems consisting of a single class of constraints, such as hypergraph colouring.

\begin{definition}
Given integers $k\ge 1, r\ge 2$ and a distribution $\pi$ over $\Lambda_{k,r}$, a random $k$-UE-SAT system, denoted by $H_n(\pi,k,m)$, is defined as follows. 
\begin{enumerate}[(i)]
    \item The set of variables is $V_n:=\{v_1,\ldots,v_n\}$;
        \item For each number $1\le j\le m$, independently choose a uniformly random $k$-tuple of distinct variables $x_1,\ldots,x_k\in \{v_1,\ldots,v_n\}$ and independently choose a constraint function $\psi\in \Lambda_{k,r}$ according to distribution $\pi$. The $j$-th constraint of $H_n(\pi,k,m)$ is defined by $(\psi,(x_1,\ldots, x_k))$, with $\psi$ being the constraint function, and $(x_1,\ldots, x_k)$ being the constraint variables.
    \end{enumerate}
    If $\pi$ is an atomic distribution where $\pi(\psi)=1$ for some $\psi\in \Lambda_{k,r}$, then we denote $H_n(\pi,k,m)$ simply by $H_n(\psi,k,m)$.
\end{definition}

\subsection{Main results}

Thoughout the paper, $\Omega$ is a finite nonempty set with $r\ge 2$ elements. 
    We refer to $\Omega$ as the spin set, and elements of $\Omega$ as spins. We denote the concatenation of two strings $\sigma_1\in \Omega^i$ and $\sigma_2\in \Omega^j$ by $\sigma_1\sigma_2$, and denote by $\sigma^j$ the concatenation of $\sigma$ with itself $j$ times. The notation $\Pi(S)$ denotes the permutation group on a set $S$. For any $\pi\in \Pi([j])$ and $\sigma=(\sigma_1,\ldots, \sigma_j)\in \Omega^j$, we write $\sigma^\pi=(\sigma_{\pi(1)},\ldots,\sigma_{\pi(j)})$.
Given a $k$-ary constraint function, and $\sigma\in\Omega^k$, we use the simpler notation $\sigma\models \psi$ to denote that $\psi(\sigma)=1$; that is, $\sigma$ satisfies $\psi$.

    \begin{definition}
        A $k$-ary constraint function $\psi$ on $\Omega$ is commutative if, for all $\sigma\in \Omega^k$ and $\pi\in \Pi([k])$,
        \begin{equation*}
            \sigma\models \psi\iff \sigma^\pi \models \psi.
        \end{equation*}
    \end{definition}

    \begin{definition}
    A $k$-ary constraint function $\psi$ is reducible if, for all $d\in \{0,\ldots, k-1\}$, $\sigma,\sigma'\in \Omega^d$, and $\eta,\eta'\in \Omega^{k-d}$ such that
    \begin{align*}
    \sigma\eta\models \psi,\quad \sigma'\eta\models\psi,\quad \sigma\eta'\models \psi,
    \end{align*}
    we also have $\sigma'\eta'\models \psi$.
    \end{definition}

We say $\psi$ admits a constant solution if $\sigma^k\models \psi$ for some $\sigma\in\Omega$. A solution to $x$ to $H_n(\psi,k,m)$ is called a constant solution if $x=\sigma^n$ for some $\sigma\in \Omega$. Note that if $\psi$ admits a constant solution $\sigma^k$ then $H_n(\psi,k,m)$ is always satisfiable for every $m$, as $x=\sigma^n$ is a solution. 
Thus, for satisfiability, it is only interesting to consider $\psi$ that does not admit a constant solution. Our next theorem confirms that the same $d_k$ in Theorem~\ref{thm:equation-sat} is the satisfiability threshold for $H_n(\psi,k,m)$ provided that $\psi$ is reducible and admits no constant solutions.

\begin{theorem}\label{thm:sat}
        Suppose that $k\geq 4$ and that $\psi\in\Lambda_{k,r}$ is reducible. Let $d_k$ be defined as in Theorem~\ref{thm:equation-sat} and let $\eps>0$.
        If $\psi$ has no constant solutions, then
        \begin{equation*}
            \lim_{n\to\infty}\pr[H_n(\psi,k,m)\ \text{is satisfiable}]=\begin{cases}
                1& \text{if}\ m<(d_k/k-\eps)n\\
                0& \text{if}\ m>(d_k/k+\eps)n.
            \end{cases}
        \end{equation*}
    \end{theorem}

Theorem~\ref{thm:sat} is a special case of a more general result (see Theorem~\ref{thm:sat2} below) which treats the broader model $H_n(\pi,k,m)$ and also covers the case $k=3$. To set the stage for this generalization, we introduce several definitions.

    For any $k$-ary UE constraint function $\psi$, let $f_\psi:\Omega^{k-1}\to \Omega$ be the unique function such that $\sigma f_\psi(\sigma)\models \psi$ for all $\sigma\in \Omega^{k-1}$. 

    \begin{definition}
        Let $\Psi\subseteq \Lambda_{k,r}$. Define a sequence of collections of functions $F_0(\Psi),F_1(\Psi),\ldots$ inductively by setting
        \begin{enumerate}[(i)]
            \item Let $F_0(\Psi)$ be the singleton set containing the identity function on $\Omega$. That is, $F_0(\Psi)$ contains a single function which maps $\sigma$ to $\sigma$ for every $\sigma\in \Omega$.
            \item For each $i\ge 1$, let $F_i(\Psi)$ be the set of functions $\Omega^{(k-1)^i}\mapsto \Omega$ of the form
            \begin{equation*}
                \sigma_1\sigma_{2}\ldots \sigma_{k-1}\mapsto f_{\psi}\bigl (f_1(\sigma_1)\ldots f_{k-1}(\sigma_{k-1})),\quad  \sigma_1,\ldots,\sigma_{k-1}\in \Omega^{(k-1)^{i-1}}
            \end{equation*}
            for some $\psi\in \Psi$ and $f_1,\ldots,f_{k-1}\in F_{i-1}(\Psi)$.
        \end{enumerate}
    \end{definition}

 We say that a $j$-ary function $f$ on a set $S$ is symmetric if $f(\sigma_1,\ldots,\sigma_j)=f(\sigma_{\pi(1)},\ldots,\sigma_{\pi(j)})$ for all $\pi\in \Pi([j])$ and $\sigma_1,\ldots,\sigma_j\in S$. 

 \begin{definition}\label{def:reducibility}
     Let $\Psi\subset \Lambda_{k,r}$. We say that $\Psi$ is reducible if, for all $d\in \{0,\ldots, k-1\}$, $\sigma,\sigma'\in\Omega^d$, $\eta,\eta'\in \Omega^{k-d}$, and $\psi,\psi'\in \Psi$ such that
     \begin{equation*}
         \sigma\eta\models\psi,\quad \sigma'\eta\models \psi,\quad \sigma\eta'\models \psi',
     \end{equation*}
     we also have $\sigma'\eta'\models \psi'$.
 \end{definition}

\begin{remark}\label{r:extension-exchange}
    We say two strings $\sigma,\sigma'$ of the same length have a common extension if there exist some string $\eta$ and $\psi\in\Psi$ such that $\sigma\eta\models \psi$ and $\sigma'\eta\models \psi$. An important implication of $\Psi$ being reducible is the following: if  $\sigma\eta'\models \psi'$ and $\sigma$ and $\sigma'$ have a common extension, then $\sigma'\eta'\models \psi'$. This ``exchangeability property'' will be useful in our proof presented later.
\end{remark}

We say that a set of constraints $\Psi\subset \Lambda_{k,r}$ has a constant solution if there is a $\sigma\in \Omega$ such that $\sigma^k\models \psi$ for all $\psi\in \Psi$. If $\supp(\pi)$ has a constant solution, then $H(\pi,k,m)$ is satisfiable for all $m$. We therefore only consider distributions $\pi$ such that $\supp(\pi)$ has no constant solutions.

\begin{theorem}\label{thm:sat2}
        Suppose that $\supp(\pi)\subseteq\Lambda_{k,r}$ satisfies one of the following two conditions.
\begin{itemize}
    \item $k\geq 4$ and $\supp(\pi)$ is reducible;
    \item $k=3$, and every function in $F_2(\supp(\pi))$ is symmetric.
\end{itemize}
Let $d_k$ be defined as in Theorem~\ref{thm:equation-sat} and let $\eps>0$. If $\supp(\pi)$ has no constant solutions, then
        \begin{equation*}
            \lim_{n\to\infty}\pr[H_n(\pi,k,m)\ \text{is satisfiable}]=\begin{cases}
                1& \text{if}\ m<(d_k/k-\eps)n\\
                0& \text{if}\ m>(d_k/k+\eps)n.
            \end{cases}
        \end{equation*}
    \end{theorem}

\begin{remark}
Coja-Oghlan et al.~\cite{coja2020replica} studied and determined the RSB threshold (also called the condensation threshold) for a general family of random CSPs that satisfy some assumptions \textbf{SYM}, \textbf{BAL}, \textbf{MIN}, \textbf{POS}, and \textbf{UNI}. The RSB threshold is in general a lower bound on the satisfiability threshold, and for random linear equations over finite fields, these two thresholds coincide, which we conjecture to hold for $H_n(\pi,k,m)$ as well; see Conjecture~\ref{conj:RSB} below. However, the results in~\cite{coja2020replica} do not apply to our model because the assumptions \textbf{BAL}, \textbf{MIN}, and \textbf{POS} do not hold for all choices of distribution $\pi$; in particular, it does not hold when $\pi$ is supported on a single $\psi$.    
\end{remark}

Next we consider the solution space of $H_n(\pi,k,m)$ for distributions $\pi$ such that $\supp(\pi)\subset \Lambda_{k,r}$. 
A sequence of works (we refer the readers to the literature reviews in~\cite{krzakala2007gibbs,dingslysun2022kSAT,achlioptas2015solution}) has revealed phenomena that appear to be ubiquitous across various random CSPs. As mentioned earlier, when the number of constraints $m$ increases, the random CSP undergoes several phase transitions, commonly known as the ``clustering'', ``freezing'', and ``RSB'' (also called the ``condensation''), before ultimately reaching the  ``satisfiability'' threshold. These phase transitions demonstrate the sudden change of the geometric relations between the solutions, and the sudden change of the correlation between two solutions sampled independently and uniformly at random from the solution space. These phase transitions help explain the failure of certain classes of algorithms below specific density thresholds -- typically around the ``clustering'' threshold -- and have motivated the development of more sophisticated methods, such as Survey Propagation. Insights from the geometry of the solution space have, in some cases, even guided the determination of the satisfiability threshold (e.g.\,~\cite{dingslysun2022kSAT,ayre2020satisfiability}). For random $k$-XORSAT, the clustering and freezing thresholds coincide~\cite{achlioptas2015solution}, as do the condensation and satisfiability thresholds~\cite{ayre2020satisfiability}. We generalize these observations to the solution space of $H_n(\pi,k,m)$. Before formally defining the clustering threshold, we introduce a few additional definitions.

Let $\sigma,\tau$ be two solutions of $H_n(\pi,k,m)$. We say $\sigma$ and $\tau$ are $\alpha$-connected if there exists a sequence of solutions $\sigma_0,\sigma_1,\ldots, \sigma_{\ell}$ of $H_n(\pi,k,m)$ such that $\sigma_0=\sigma$, $\sigma_{\ell}=\tau$ and $|\sigma_{i}\Delta \sigma_{i+1}|\le \alpha$ for every $0\le i<\ell$. Let $S$ and $S'$ be two disjoint subsets of solutions of $H_n(\pi,k,m)$. We say that $S$ and $S'$ are $\beta$-separated if $|\sigma\Delta \sigma'|\ge \beta$ for every $\sigma\in S$ and $\sigma'\in S'$.

Recalling $\rho_{k,d}$ from Theorem~\ref{thm:equation-sat},
%\begin{equation*}
 %   \rho_{k,d} = \sup\lt\{x\in [0,1]: x=\exp(-dx^{k-1})\rt\}.\\
%\end{equation*}
define $d_k^\star$ by
\begin{equation*}
    d_k^\star=\inf \{d>0: \rho_{k,d}>0\}.
\end{equation*}

Given a distribution $\pi$ over $\Lambda_{k,r}$, let 
$$Q_{\pi}=\{\sigma\in \Omega: \sigma^k\models \psi, \forall \psi\in \supp(\pi)\}.
$$

\begin{theorem}\label{thm:clustering} (Clustering threshold) Suppose that $\supp(\pi)\subseteq\Lambda_{k,r}$ satisfies one of the following two conditions.
\begin{itemize}
    \item $k\geq 4$ and $\supp(\pi)$ is reducible;
    \item $k=3$, and every function in $F_2(\supp(\pi))$ is symmetric.
\end{itemize}
  Let $\eps>0$.
 \begin{enumerate}[(a)]
        \item \label{thm:clustering-a}If $m<(d_k^\star/k-\eps)n$, then a.a.s.\ all solutions in $H_n(\pi,k,m)$ are pairwise $O(\log n)$-connected. 
        \item \label{thm:clustering-b}If $(d_k^\star/k+\eps)n<m<(d_k/k-\eps)n$, then with $d=km/n$, a.a.s.\ the solutions of $H_n(\pi,k,m)$ are partitioned into clusters $C_1,\ldots,C_{N}$ for some integer $N$ such that
        \begin{enumerate}[{(i)}]
        \item \label{thm:clustering-b-i}$n^{-1}\log N \to (\rho_{k,d}-d\rho_{k,d}^{k-1}+(d-d/k)\rho_{k,d}^k)\log r$;
            \item \label{thm:clustering-b-ii}$|C_i|=|C_j|$ for every $1\le i<j\le N$, and for every $1\le i\le N$:
            $$n^{-1}\log |C_i| \to (1-d/k-\rho_{k,d}+d\rho_{k,d}^{k-1}-(d-d/k)\rho_{k,d}^k)\log r;$$
            \item \label{thm:clustering-b-iii}Solutions in the same cluster $C_i$ are pairwise $O(\log n)$-connected;
            \item \label{thm:clustering-b-iv} Every pair of distinct clusters are $\Omega(n)$-separated.
        \end{enumerate}
         \item\label{thm:homogeneousclustering} If $m>(d_k/k+\eps)n$ and $Q_{\pi}\neq\emptyset$, then a.a.s.\ the solutions of $H_n(\pi,k,m)$ are partitioned into $|Q_{\pi}|$ clusters that satisfy properties (ii)--(iv) in part (b). 
        \end{enumerate}
\end{theorem}

\begin{remark}
    Note that if $Q_{\pi}\neq \emptyset$ then $\pi$ is supported on a single $\psi$ by the reducibility of $\supp(\pi)$ if $k\ge 4$ or the symmetry of $F_2(\supp(\pi))$ if $k=3$, which is assumed by the hypothesis of the theorem. We leave this observation as a simple exercise. Our proof of Theorem~\ref{thm:clustering} does not rely on this fact.
\end{remark}

We expect that the phase transition of RSB occurs when $m/n$ is around $d_k/k$ if $\supp(\pi)$ satisfies the same conditions as in Theorem~\ref{thm:clustering}. Indeed, it follows easily from our results in this paper (see Remark~\ref{r:RSB} below) when $r$ is the product of distinct prime numbers. However, for general $r$ we did not manage to prove that the replica symmetry holds until reaching the satisfiability threshold, and thus we leave it as a conjecture; see Conjecture~\ref{conj:RSB}.

%    A $k$-ary constraint function on $\Omega$ is a $k$-ary function $\psi:\Omega^k\to\{0,1\}$. We say that a string $\sigma\in\Omega^k$ satisfies $\psi$, and write $\sigma\models \psi$, if $\psi(\sigma)=1$. A CSP instance is a tuple $H=(V,C,\{\psi_a\}_{a\in C},\{\partial_a\}_{a\in C})$, where $V$ is a variable set, $C$ is a index set for the constraints in $H$ (we may take $C=[m]$ if $H$ contains $m$ constraints), $\psi_a$ is a constraint function for each $a\in C$, and $\partial_a=(\partial_{a,1},\ldots, \partial_{a,k})\in V^k$ for each $a\in C$. We say that a variable assignment $\sigma:V\to \Omega$ satisfies $H$, and write $\sigma\models H$, if $\sigma(\partial_{a,1})\cdots \sigma(\partial_{a,k})\models \psi_a$ for all $a\in C$. \jc{keep the following for now.} We say $x=(x_v)_{v\in V}\in \Omega^V$ is a solution of $H$ if the assignment $\sigma$ given by $\sigma(v)=x_v$ for all $v\in V$ satisfies $H$.
    
%    Let $V$ be a countably infinite variable set with some fixed enumeration, and let $V_n$ be the first $n$ elements of $V$. 

\remove{

%%%%%%%%%%%%
%%%%%%%%%%%%

Given a fixed spin set $\Omega$, and integers $k,m\geq 1$, and a constraint function $\psi$ on $\Omega$, we choose a random CSP instance $H_n(\psi,k,m)$ as follows:
    \begin{enumerate}[(i)]
    \item The set of variables is $V_n:=\{v_1,\ldots,v_n\}$;
        \item For each number $1\le j\le m$, independently choose a uniformly random $k$-tuple of distinct variables $x_1,\ldots,x_k\in \{v_1,\ldots,v_n\}$ and let $\partial_j=(x_1,\ldots,x_k)$.
        Let 
        $$H_n(\psi,k,m)=(V_n,[m], \{\psi_a\}_{a\in[m]},\{(\partial_{j,1},\ldots, \partial_{j,k})\}_{j\in[ m]}),$$ were $\psi_a=\psi$ for all $a\in [m]$.
%        $(\partial_{j,1},\ldots, \partial_{j,k})\in \Omega^k$.
%        \item Choose a Poisson random variable $\rv m$ with mean $dn/k$ independently of $\{(\partial_{j,1},\ldots, \partial_{j,k})\}_{j\in\NN}$, and let $H_n(\psi,k,d)=(V_n,[\rv m], \{\psi_a\}_{a\in[\rv m]},\{(\partial_{j,1},\ldots, \partial_{j,k})\}_{j\in[\rv m]})$, were $\psi_a=\psi$ for all $a\in [\rv m]$. 
    \end{enumerate}
%    When $\psi$, $k$, and $m$ are implicit, we write $H_n=H_n(\psi,k,d)$. 

%%%%%%%%%%%%%%
%%%%%%%%%%%%%%
}

%Note that every constraint in our model uses the same constraint function. This is in contrast with the model used in ***. \jc{I will comment on this later.}

\remove{%%%%%%%%%%
%%%%%%%%%%%%

    Our focus will be on constraint functions that are commutative and uniquely extendable.

    \begin{definition}
        A $k$-ary constraint $\psi$ on $\Omega$ is uniquely extendable (UE) if, for all $j\in [k]$, $\sigma_1\in \Omega^{j-1}$, and $\sigma_2\in \Omega^{k-j}$, there is a unique $\tau\in \Omega$ such that $\sigma_1\tau\sigma_2\models \psi$.
    \end{definition}
%%%%%%%%%%%
%%%%%%%%%%
}

\subsection{An overview of the proofs}

Let $(G,+)$ be an abelian group of order $r\ge 2$. Consider $\Lambda'\subseteq \Lambda_{k,r}$ defined by $\Lambda'=\{\psi_{G,b}\}_{b\in G}$ where
\begin{equation}
        \sigma\models \psi_{G,b}\iff\sum_{i=1}^k \sigma_i=b,\quad \sigma=(\sigma_1,\ldots, \sigma_k)\in G^k. \label{def:psi-b}
\end{equation}
Ayre et al.\,~\cite{ayre2020satisfiability} proved that Theorem~\ref{thm:sat2} holds if $\pi$ is the uniform distribution over $\Lambda'$. Our first step is to show that Theorem~\ref{thm:sat2} holds for $H_n(\pi,k,m)$ for an arbitrary distribution on $\Lambda'$, provided that $\supp(\pi)$ does not admit a constant solution. In particular, the theorem holds for $H_n(\psi_{G,b},k,m)$ for every $b\in G$ whenever $\psi_{G,b}$ does not have a constant solution. 

\begin{theorem}\label{thm:sat3}
    Let $(G,+)$ be an abelian group of order $r\ge 2$ and $b\in G$. Let $\psi_{G,b}$ be defined as in~\eqref{def:psi-b}. Let $\pi$ be a distribution over $\{\psi_{G,b}\}_{b\in G}$ where $\supp(\pi)$ admits no constant solutions. Suppose that $k\ge 3$ and $d_k$ is defined as in Theorem~\ref{thm:equation-sat}. 
        Then for any $\eps>0$,
        \begin{equation*}
            \lim_{n\to\infty}\pr[H_n(\pi,k,m)\ \text{is satisfiable}]=\begin{cases}
                1& \text{if}\ m<(d_k/k-\eps)n\\
                0& \text{if}\ m>(d_k/k+\eps)n.
            \end{cases}
        \end{equation*}
\end{theorem}
Next, we identify UE constraint functions that are equivalent to functions in~\eqref{def:psi-b} for some abelian group $(G,+)$. 

    We say that a $k$-ary UE constraint function $\psi$ on $\Omega$ is a equivalent to a \emph{group constraint function} if there exists a (not necessarily abelian) group $(\Omega,+)$ such that $\psi=\psi_{(\Omega,+),b}$ for some $b\in \Omega$.

Now Theorems~\ref{thm:sat} and~\ref{thm:sat2} follow from Theorem~\ref{thm:sat3} and the following two theorems that characterise UE commutative constraint functions that are equivalent to group constraint functions.

    \begin{theorem}\label{symthm}
    Suppose $k\geq 3$ and suppose $\Psi\subset \Lambda_{k,r}$ is such that every $\psi\in \Psi$ is commutative. Then the following are equivalent.
    \begin{enumerate}[(a)]
        \item \label{symthm1}There is a group $(\Omega,+)$ on the spin set $\Omega$ and a collection of elements $\{b_{\psi}\}_{\psi\in \Psi}\subset  \Omega$ such that $\psi= \psi_{(\Omega,+),b_\psi}$ for all $\psi\in \Psi$.
        \item \label{symthm2}Every function in $F_2(\Psi)$ is symmetric.
    \end{enumerate} 
\end{theorem}

    \begin{theorem}\label{cor:reducible}
    Suppose $k\geq 4$ and suppose $\Psi\subset \Lambda_{k,r}$ is such that every $\psi\in \Psi$ is commutative. Then the following are equivalent.
    \begin{enumerate}[(a)]
        \item There is a group $(\Omega,+)$ on the spin set $\Omega$ and a collection of elements $\{b_{\psi}\}_{\psi\in \Psi}\subset  \Omega$ such that $\psi= \psi_{(\Omega,+),b_\psi}$ for all $\psi\in \Psi$.
        \item The set of constraints $\Psi$ is reducible.
    \end{enumerate} 
    \end{theorem}

  \begin{remark} 
  \begin{enumerate}[(a)]
      \item    Determining whether an arbitrary UE commutative constraint is a group constraint requires \emph{prima facie} exponentially many satisfaction queries in both $r=|\Omega|$ and $k$, as one would need to check every permutation of $\Omega$ (even if one knew the isomorphism class), and thus the complexity is approximately $r!r^k$. However, checking reducibility of a given $k$-ary UE constraint function can be done more efficiently with complexity approximately $kr^{2k}$. 
    
  \item  The assumption $k\geq 4$ in Theorem \ref{cor:reducible} is necessary. In fact, it is easy to show that every $3$-ary UE commutative constraint function is reducible (we leave this as an easy exercise to the reader). However, not all these constraint functions are equivalent to group constraint functions. A concrete example for $(k,q)=(3,4)$ was given by Connamacher and Molloy where $\Psi$ contains three constraint functions, each as a group constraint, but $F_2(\Psi)$ is not symmetric. We provide a family of $3$-ary UE commutative constraint functions that are not equivalent to any group constraint function in section \ref{sec:rediciblethm}. The same construction indeed also provides a family of $k$-ary non-reducible UE constraint functions for $k\ge 4$.

\item The special case of $\Psi=\{\psi\}$ of Theorems~\ref{symthm} and~\ref{cor:reducible} can be proved using known results from quasigroup theory, and we sketch these arguments in the Appendix. It is possible that the quasigroup results we rely on can be further generalised to handle the full versions of Theorems~\ref{symthm} and~\ref{cor:reducible}. We have not pursued this direction. Instead, we provide self-contained proofs of these theorems without using quasigroup theory, and we believe that these proofs are of independent interest.
  \end{enumerate}

 \end{remark}   

\subsection{More conjectures and future work}

While Conjecture~\ref{conj:Molloy} remains open, determining the threshold of satisfiability for $H_n(\pi,k,m)$ is of independent interest for many distributions $\pi$. This is especially interesting when the support of $\pi$ corresponds to a natural random system of UE-SAT system, such as linear equations over a field. We propose a few conjectures and open problems related to Conjecture~\ref{conj:Molloy}.

\remove{
%%%%%%%
%%%%%%%
Our first conjecture is a natural generalisation of Theorem~\ref{thm:sat}. Note that the satisfiability threshold of $H_n(\psi,k,m)$ in Theorem~\ref{thm:sat} is independent of $\psi$. It is therefore plausible to conjecture that the same threshold holds if $\psi$ is chosen randomly according to a distribution $\mu$ supported on reducible members of $\Lambda_{k,r}$, provided that $\supp(\mu)$ contains some $\psi$ that does not accept a constant solution.

\begin{conjecture}
            Suppose that $k\geq 3$ and that $\mu $ is a distribution over $\Lambda_{k,r}$ satisfying
            \begin{itemize}
                \item all functions in $\supp(\mu)$ are commutative and reducible if $k\ge 4$; or
                \item all functions  $\psi\in \supp(\mu)$ are commutative with symmetric $f_{\psi}^{(2)}$, if $k= 3$.
            \end{itemize}
            Let $d_k$ be defined as in Theorem~\ref{thm:equation-sat}.
        If $\mu$ admits no constant solutions, then $d_k/k$ is the satisfiability threshold (for $d$ where $d=m/n$) of $H_n(\mu,k,m)$. 
\end{conjecture}

The proof of Theorems~\ref{thm:sat} and~\ref{thm:sat2} apply Theorem~\ref{thm:sat3} only for atomic distributions $\mu$ supported on a single $\psi_{G,b}$. To strengthen these results and fully exploit Theorem~\ref{thm:sat3} \jc{Theo knows how to do it, probably}, it is natural to consider distributions $\pi$ over $\Lambda_{k,r}$ supported on sets of commutative UE constraint functions that are equivalent to group constraint functions on some abelian group $(G,+)$. This raise the following question.  

\begin{problem} Let $k\ge 3$.
    Given a set $S$ of $k$-ary commutative UE constraint functions, is there a characterisation, or an efficient algorithm, that determines whether there exists an abelian group $(G,+)$ such that every member in $S$ is equivalent to a group constraint function on $(G,+)$? 
\end{problem}

The model $H_n(\mu,k,m)$ in Theorem~\ref{thm:sat3} is a random UE-SAT system arising from random equations over finite abelian groups. We conjecture that the same result holds if the abelian condition is removed.
}
Our first conjecture is a natural generalisation of Theorem~\ref{thm:sat3} by removing the abelian condition.
\begin{conjecture}
    Theorem~\ref{thm:sat3} holds if $G$ is a finite group of order at least two.
\end{conjecture}

Our second conjecture is a generalisation of Theorem~\ref{thm:sat2} which weakens the condition that $\supp(\pi)$ is reducible. For a single UE constraint $\psi$, the sets $F_i(\{\psi\})$ are singletons for all $i\geq 0$.  Let $f^{(i)}_\psi$ denotes the unique function in $F_i(\{\psi\})$. Given a set $\Psi\subseteq \Lambda_{k,r}$, if $f_{\psi}^{(2)}$ is symmetric for every $\psi\in \supp(\pi)$, it does not necessarily follow that $F_2(\Psi)$ is symmetric.

\begin{conjecture}
Theorem~\ref{thm:sat2} holds if $\supp(\pi)$ satisfies one of the two following conditions.
\begin{itemize}
    \item $k\ge 4$ and every $\psi\in \supp(\pi)$ is reducible.
    \item $k=3$ and $f_{\psi}^{(2)}$ is symmetric for every $\psi\in \supp(\pi)$ .
\end{itemize}
\end{conjecture}

\noindent{\bf Note.} The setting of this conjecture permits a random system of group equations of form~\eqref{def:psi-b}, where different equations can be over different groups, provided that all the groups have the same order. This flexibility is motivated by the fact that the SAT threshold in Theorem~\ref{thm:sat3} is independent of the group structure of $G$. However, our current proof crucially relies on the assumption that all equations are defined over a single common group. 

Finally, we make the following conjecture on the phase transition of replica symmetry breaking at $d_k$. 
It has been shown that the ``replica symmetry breaking'' (or ``condensation'') occurs at density $d_k$ for random linear equations over finite fields.  We conjecture that the same holds for $H_n(\pi,k,m)$, at least when $\supp(\pi)$ is reducible. Recall the definition of $Q_{\pi}$ before Theorem~\ref{thm:clustering}.

\begin{conjecture}\label{conj:RSB}
Suppose that $\supp(\pi)$ satisfies one of the two following conditions.
\begin{itemize}
    \item $k\geq 4$ and $\supp(\pi)$ is reducible;
    \item $k=3$, and every function in $F_2(\supp(\pi))$ is symmetric.
\end{itemize}
Let $\eps>0$.

\begin{enumerate}[(a)]
    %\item If $d<d_k$, and $x$ and $y$ are two independent uniform random solutions of $H_n(\psi,k,m)$, then a.a.s.
   % \[
    %|x^{-1}(\sigma)\cap y^{-1}(\tau)-n/q^2|=o(n),\quad \text{for all $\sigma,\tau\in Q$}.
    %\]
    %\item If $d>d_k$ and
    %$Q_{\psi}\neq\empty$ where $Q_{\psi}$ is the set of constant solutions of $\psi$, then, for two independent uniform random solutions of $H_n(\psi,k,m)$  $x$ and $y$, 
%with probability $1/h$, 
%??, and with probability $1-1/h$,
%??.

    \item If $m<(d_k/k-\eps)n$, $x$ is a uniform random solution of $H_n(\pi,k,m)$, and $u,v$ are uniform random variables in $V_n$, where $x,u,v$ are mutually independent, then 
    \begin{equation}
\pr[x_u=\sigma, x_v=\tau] =1/r^2+o(1),\quad \text{for all $\sigma,\tau\in \Omega$}. \label{eq:RS}
    \end{equation}
    \item If $m>(d_k/k+\eps)n$ and $Q_{\pi}\neq \emptyset$, then, for a uniform random solution $x$ of $H_n(\pi,k,m)$, property~\eqref{eq:RS} does not hold a.a.s.
\end{enumerate}
\end{conjecture}
\begin{remark}\label{r:RSB}
Part (b) of the conjecture follows easily from Our Theorem~\ref{thm:clustering} (c).
By our Theorems~\ref{symthm},~\ref{cor:reducible}, and the RS property of the solution space of random linear equations over finite fields~\cite[eq.~(2.5)]{ayre2020satisfiability}, part (a) of the conjecture follows immediately if $r$ is a product of distinct prime numbers. However, this approach of reduction does not work if for instance $r=p^k$ for some prime $p$ and $k\ge 2$, since the solutions of linear equations over $GF(r)$ is not bijectively mapped to the solutions over ${\mathbb Z}_r$.

\end{remark}

The structure of the rest of the paper is organised as follows. We prove Theorem~\ref{thm:sat3} in section~\ref{sec:satthm}. Then,
in section~\ref{sec:symmetricthm}, we prove Theorem~\ref{symthm} and Theorem~\ref{cor:reducible}. The proof for Theorem~\ref{thm:clustering}  is presented in Section~\ref{sec:clustering}. In section \ref{sec:rediciblethm},  we exhibit a family of UE-constraint functions that are not equivalent to group constraint functions for $k\geq 3$. 
In the Appendix, we sketch alternative proofs for special cases of Theorems~\ref{symthm} and~\ref{cor:reducible} when $\supp(\pi)=\{\psi\}$, and demonstrate their connections to quasigroup theory.

\section{Proof of Theorem~\ref{thm:sat3}}\label{sec:satthm}

In general,
given $H=H_n(\pi,k,m)$, and $1\le a\le m$, let $\psi_a$ denote the $a$-th constraint function, and let  $\partial_a=(\partial_{a,1},\ldots, \partial_{a,k})\in V_n^k$ denote its constraint variables. Hence, the $a$-th constraint is given by $(\psi_a,\partial_a)$.

%Given an instance $H$ from the sample space of $H_n(\psi,k,m)$, the Tanner graph of $H$ is a bipartite graph defined on $V_n\cup [m]$ where for every $1\le a\le m$, $a\in [m]$ is adjacent to the set of vertices in $\partial_a$; i.e.\ the constraint variables of the $a$-th constraint in $H$. 
Given an instance $H$ from the sample space of $H_n(\pi,k,m)$, let $B=B(H)$ denote the $m\times n$ binary matrix defined by $B_{a,v}=\ind{v\in \partial_a}$. Then, $B(H_n(\pi,k,m))$ is a random $m\times n$ matrix in which each row has exactly $k$ nonzero entries equal to 1, whose positions are chosen uniformly at random from $[n]$, and choices are independent for each row. If there exists an abelian group $(G,+)$ such that every constraint function in $H$ is equivalent to $\psi_{G,g}$ for some $g\in G$, then, let $b=b(H)\in G^m$ be defined by $b_a=c$, where $\psi_a=\psi_{G,c}$. This way, the set of solutions to $H$ is exactly the set of solutions to $Bx=b$.

The 2-core of $B$, denoted by $B_{\text{2-core}}$, is defined to be the maximum submatrix of $B$ such that every column has at least two nonzero entries. The threshold for the appearance of, and the size of $B_{\text{2-core}}$ have been well understood. For our proof, we use the following lemma from~\cite[Lemmas 12 and 14]{gao2023minors}.

\begin{lemma}\label{indeprows} Let $(k,r)\ge (3,2)$ and $\psi\in\Lambda_{k,r}$. Let $\FF$ be a finite field. There is  $\eps_n=o(1)$ such that if $m=(d_k/k-\eps_n)n$ then the following hold a.a.s.\ for $B=B(H_n(\psi,k,m))$.
\begin{enumerate}[(a)]
    \item The rows of $B$ are linearly independent over $\FF$.
    \item With $n^*$ and $m^*$ denoting the number of columns and rows in $B_{\text{2-core}}$, $n^*=\Theta(n)$ and $m^*=(1+O(\eps_n))n^*$.
\end{enumerate}
\end{lemma}

%\begin{theorem}
 %   Suppose $\psi$ is a $k$-ary reducible commutative UE constraint with $k\geq 4$. If $\psi$ does not admit a constant solution, then
 %   \begin{equation*}
  %      \lim_{n\to\infty}P[\exists x\in \Omega^n:x\models B_n]=\begin{cases}
   %         1& d<d_k\\
   %         0&d>d_k
   %     \end{cases}
   % \end{equation*}
   % where $d_k$ is the same as for fields.
%\end{theorem}
\begin{proof}[Proof of Theorem~\ref{thm:sat3}]
    By the fundamental theorem of finite Abelian groups, it suffices to consider the case where $\psi$ is equivalent to cyclic group $\ZZ_{q}$, where $q$ is a power of some prime $p$.

The proof for the subcritical case is identical to the treatment in ~\cite{ayre2020satisfiability}.
Suppose $m<(d_k/k-\eps)n$. Then Theorem 1.1 in~\cite{ayre2020satisfiability} shows that $B=B(H_n(\pi,k,m))$ has full row rank as a matrix over $\FF_p$ a.a.s.\ Suppose that $B$ has full row rank over $\FF_p$. Then there is an $m\times m$ submatrix $B'$ of $B$ such that the determinant $\det(B')$ of $B'$ as an integer matrix satisfies $\det(B')\neq 0\pmod p$. Therefore $\det(B')\pmod q$ is an invertible element of the ring $\ZZ_q$, and so $B'$ is invertible over $\ZZ_q$. Therefore the number of solutions to $Bx=b$ is $q^{n-m}>0$. Therefore $H_n(\pi,k,m)$ is satisfiable a.a.s.\

Next, suppose that $m>(d_k/k+\eps)n$.
Let $v\in \{0,1\}^n$ be chosen uniformly from the set of binary strings with length $n$ and containing precisely $k$ 1's. Let $v_1,\ldots,v_{m}$ be independent copies of $v$.
Then, $B=B(H_n(\psi,k,m))$ has the same distribution as the matrix with $m$ rows, with $v_j^T$ being the $j$-th row for $1\le j\le m$. Let $[n]$ and $[m]$ denote the set of columns and rows of $B$.

Let $\eps_n$ be chosen to satisfy Lemma~\ref{indeprows} for the finite field $\FF_p$, and let $m'=(d_k/k-\eps_n)n$ and $m''=m-m'$. Let $B'$ be the submatrix of $B$ consisting of the first $m'$ rows, and $B''$ be the submatrix of $B$ consisting of the last $m''$ rows.
 Given a subset $I$ of rows of $B$, let $B|_I$ and $b|_{I}$ be the matrix and vector obtained by restricting $B$ and $b$ to the rows in $I$ respectively. Let $R_n$ and $C_n$ be the rows  and columns respectively of $B'_{\text{2-core}}$, i.e.\ the 2-core of $B'$.  
\begin{claim}\label{claim:top-system}
    Let $\Sigma$ denote the set of solutions to $B'_{\text{2-core}}\,x=b|_{R_n}$. Then, a.a.s.\ $|\Sigma|=q^{O(\eps_n n)}$.
\end{claim}
Let $R'_n$ be the set of rows in $B''$ such that each row in $R'_n$ is supported only on $C_n$; i.e.\ row $i$ is in $R'_n$ if it is a row in $B''$, and all entries outside $C_n$ are zero. By Lemma~\ref{indeprows}(b) and the Chernoff bounds,
\begin{equation}
\text{a.a.s.}\ |R'_n|=\Omega(m''). \label{eq:size-of-R'}
\end{equation}
\begin{claim}\label{claim:bottom-system} Condition on $R'_n$.
There exists a fixed $\delta>0$ such that
    for every row $i\in R'_n$ and every $x\in\Sigma$, 
    $$\pr[B|_{i}\, x = b|_{i}]\le 1-\delta.$$
\end{claim}
We complete the proof of the theorem assuming Claims~\ref{claim:top-system} and~\ref{claim:bottom-system}. Since rows in $B''$ are independent, restricting their support to lie within $C_n$ and conditioning on $R'_n$ preserves their independence. By~\eqref{eq:size-of-R'} and Claim~\ref{claim:bottom-system}, for every $x\in \Sigma$,
\[
\pr[B|_{R'_n}\, x = b|_{R'_n}]\le (1-\delta)^{\Omega(m'')} \le \exp(- \Omega(\delta(\eps+\eps_n)n)).
\]
Assume that $|\Sigma|=q^{O(\eps_n n)}$ which holds a.a.s.\ by Claim~\ref{claim:top-system}. Then, the probability that there exists $x\in \Sigma$ such that $B|_{R'_n}\, x = b|_{R'_n}$ is $q^{O(\eps_n n)} \exp(- \Omega(\delta(\eps+\eps_n)n))=\exp(-\Omega(\eps n))$, as $\eps,\delta>0$ are fixed and $\eps_n=o(1)$. By Markov's inequality, a.a.s.\ $Bx=b$ is unsatisfiable.

It only remains to prove Claims~\ref{claim:top-system} and~\ref{claim:bottom-system}.

\begin{proof}[Proof of Claim~\ref{claim:top-system}]
    
     By Lemma \ref{indeprows}, the rows of $B'$ are linearly independent a.a.s.\ over $\FF_p$. By the definiton of the 2-core of $B'$, each row in $R_n$ is supported on $C_n$. Consequently, we may assume the  event that the rows of $B'_{\text{2-core}}$ are linearly independent over $\FF_p$, which holds a.a.s.  Then there is an $|R_n|\times |R_n|$ submatrix $D$ of $B'_{\text{2-core}}$ that is invertible over $\FF_p$. Similarly as in the subcritical case, we can conclude that $D$ is invertible over $\ZZ_q$. Therefore the number of solutions to $B'_{\text{2-core}}\, x=b|_{R_n}$ is $|\ZZ_q|^{|C_n|-|R_n|}$, which is $q^{O(\eps_nn)}$ by Lemma~\ref{indeprows}.
\end{proof}

\begin{proof}[Proof of Claim~\ref{claim:bottom-system}]
Fix $i\in R'_n$ and $x\in \Sigma$. 
Since every variable $x_v$ can take at most $q$ values, we immediately have the following obvious observation: 
\begin{equation}\label{obs}
\text{there exists $r\in \FF_q$ such that $|C_n(x,r)|\ge |C_n|/q$ where $C_n(x,r)$ is defined by $\{v\in C_n: x_v=r\}$.}
\end{equation}
 Since $\pi$ does not admit a constant solution, there is some $\phi\in \supp(\pi)$ that is not satisfied by $r^k$. Recall that $\partial i$ denotes the set of constraint variables in the $i$-th constraint, which is exactly the set of columns in $B|_i$ that takes nonzero values. We obtain the following which immediately implies our claimed assertion
   \begin{equation}
        \pr[B|_i\, x\neq  b|_i]\geq \pr[\psi_i=\phi, \partial i\subseteq C_n(x,r)]=\pr[\psi_i=\phi]\pr[ \partial i\subseteq C_n(x,r)]=\Omega(1),\label{eq:bad}
    \end{equation}
 where the last equation above holds, since $\pr[\psi_i=\phi]=\Omega(1)$ as $\phi\in\supp(\pi)$, and $\pr[ \partial i\subseteq C_n(x,r)]=\Omega(1)$ follows by~\eqref{obs}.
\end{proof}
        \end{proof}

\section{Proof of Theorem~\ref{thm:clustering}} \label{sec:clustering}

The 2-core of a matrix $B$ defined in section \ref{sec:satthm} can be equivalently described as the output of an iterative peeling algorithm. At each iteration of the algorithm, remove all columns of $B$ with fewer than $2$ nonzero entries as well as any rows that have nonzero entries in those columns. The algorithm terminates when all columns have at least two nonzero entries, and the resulting matrix is the 2-core of $B$.

The proofs of parts (a,b) are modifications of the proofs in~\cite{achlioptas2015solution}. We ketch briefly the main ideas and how to adjust the proofs of~\cite{achlioptas2015solution}.

\subsection{Proof of Parts (a,b)}

The following definition come from~\cite[Definitions~4 and 7]{achlioptas2015solution}, phrased slightly differently for our setting.

\begin{definition}
    A filppable cycle of $B'=B_{\text{2-core}}$ is a sequence of columns $v_0,\ldots, v_{\ell-1}$ in $B'$ for some $\ell \geq 3$ such that there is a sequence of rows $a_0,\ldots, a_{\ell-1}$ of $B'$ satisfying the following for all $i=0,\ldots, \ell-1\pmod \ell$: $B'_{a_i,v_i}=B'_{a_{i},v_{i+1}}=1$ and $B'_{a,v_i}=0$ for all rows $a$ in $B'$ except for $a=a_{i-1},a_{i}$. 
\end{definition}

\begin{definition}
     Let $V^*$ be the columns in $B_{\text{2-core}}$ that are not contained in a flippable cycle. We say that two solution $y,y'$ to $Bx=b$ are cycle equivalent if $y_v=y'_v$ for all $v\in V^*$.
\end{definition}

Now the set of solutions of $H_n(\pi,k,m)$ is partitioned into cycle equivalent classes. We call each each cycle equivalent class of solutions a cluster. Our part (a) and part (b.iii) and (b.iv) follow by exactly the same proof as~\cite[Theorems~2 and~3]{achlioptas2015solution}, by noticing that their proof only relies on the property that boolean equations are uniquely extendable.

For (b.i) and (b.ii), we count the number of clusters and the size of each cluster. We will use the following result from~\cite{achlioptas2015solution}.
\begin{lemma}[\cite{achlioptas2015solution}, Lemma 35]\label{lem:cyclebound}
    Suppose $k\geq 3$ and $m=dn/k$ for some $d>d_k^\star$. Then the expected number of vertices in flippable cycles of $B$ is $O(1)$.
\end{lemma}
First consider the case where $Bx=b$ is a system of equations over a cyclic group $\ZZ_q$, where $q$ is a power of some prime $p$. By definition of the solution clusters and Lemma~\ref{lem:cyclebound}, the number of clusters $N$ is the number of solutions to $B_{2-\text{core}}x=b$ up to a factor of $q^{\log n}$, as two such solutions might be contained in the same cluster, if they differ only on the set of variables that are contained in flippable cycles of $B_{\text{2-core}}$ --- by Lemma~\ref{lem:cyclebound} there are a.a.s.\ at most $\log n$ such variables. Let $m^*\times n^*$ be the dimensions of $B_{\text{2-core}}$. If we consider $B_{2-\text{core}}$ as a binary matrix over $\FF_p$, Lemma \ref{indeprows} shows that the rows of $B_{\text{2-core}}$ are linearly independent. There is therefore an $m^*\times m^*$ binary invertible submatrix $D$ of $B_{\text{2-core}}$. Since $\det(D)\neq 0 \pmod p$, we have $\det(D)\neq 0\pmod q$ and thus $D$ is invertible over $\ZZ_q$. Therefore the number of $2$-core solutions is $q^{n^*-m^*}$, and thus $N=q^{n^*-m^*+O(\log n)}$, where $N$ is the number of clusters. By~\cite[Theorem 1.6]{ayre2020satisfiability} for $n^*/n$ and $m^*/m$, we get
\begin{equation}\label{eq:N}
    \frac{1}{n}\log N=(n^*/n-m^*/m)\log q+O(\log n/n)\to (\rho_{k,d}-d\rho_{k,d}^{k-1}+(d-d/k)\rho_{k,d}^k)\log q.
\end{equation}
Moreover, the cluster of a solution $x$ is the set of solutions $x+z$, were $z$ satisfies $Bz=0$ and $z_v=0$ for all 2-core variables $v$ that are not contained in a flipable cycle. Thus every cluster of solutions has the same size, and the size of a cluster can be found by dividing the  the number of solutions to $Bx=b$ by $N$. By Lemma \ref{indeprows}, the rows of $B$ are a.a.s.\ linearly independent, and so we can find an $m\times m$ invertible submatrix $D'$ of $B$ over $\FF_p$. This matrix is invertible over $\ZZ_q$ as $\det(D')\neq 0\pmod p$, so the number of solution to $Bx=b$ is $q^{n-m}$ a.a.s. Thus
\begin{equation*}
    \frac{1}{n}\log|C_i|=(1-d/k)\log q-\frac{1}{n}\log N\to(1-d/k-\rho_{k,d}+d\rho_{k,d}^{k-1}-(d-d/k)\rho_{k,d}^k)\log q.
\end{equation*}

Finally, when $Bx=b$ is a system of equations over a finite abelian group $G$, we can write $G$ as $G\cong \ZZ_{q_1}\times \ldots\times \ZZ_{q_h}$ for some prime powers $q_1,\ldots, q_h$. Let $\pi_i:G\to \ZZ_{q_i}$ be the projection map for each $i\in[h]$, let $N_i$ be the number of clusters of the projected system $B\pi_i(x)=\pi_i(b)$. Since each system $B\pi_i(x)=\pi_i(b)$ has the same 2-core, solutions $y,y'$ of $Bx=b$ are in the same cluster if and only if $\pi_i(y)=\pi_i(y')$ are in the same cluster for all $i\in[h]$. Therefore the number of clusters of $Bx=b$ is the product of the number of clusters for $B\pi_i(x)=\pi_i(b)$ over $i\in[h]$, and we get
\begin{equation*}
    \frac{1}{n}\log N=\sum_{i=1}^h \frac{1}{n}\log N_i\to (\rho_{k,d}-d\rho_{k,d}^{k-1}+(d-d/k)\rho_{k,d}^k)\sum_{i=1}^h\log q_i=(\rho_{k,d}-d\rho_{k,d}^{k-1}+(d-d/k)\rho_{k,d}^k)\log r.
\end{equation*}
 Moreover, the size of the cluster of some solution $\bar x$ is the product the the sizes of the clusters of $B\pi_i(x)=\pi_i(b)$ for $\pi_i(\bar x)$ over $i\in[h]$, and so
\begin{equation*}
    \frac{1}{n}\log|C_i|\to(1-d/k-\rho_{k,d}+d\rho_{k,d}^{k-1}-(d-d/k)\rho_{k,d}^k)\log r.\qed
\end{equation*}
%For part~\ref{thm:clustering-b}~\ref{thm:clustering-b-iii}, note that solutions $y,y'$ of $Bx=b$ are $O(\log n)$ connected if and only if $\pi_i(y)$ and $\pi_i(y')$ are $O(\log n)$ connected for all $i\in[h]$, as $h$ is fixed. Therefore two solutions $y,y'$ in the same cluster are $O(\log n)$-connected. Similarly, for part~\ref{thm:clustering-b}~\ref{thm:clustering-b-iv}, solutions $y,y'$ of $Bx=b$ are $\Omega(n)$-separated if and only if there is an $i\in[n]$ such that $\pi_i(y)$ and $\pi_i(y')$ are $\Omega(n)$-separated. Therefore every pair of distinct clusters is $\Omega(n)$ separated.

\subsection{Proof of part \ref{thm:homogeneousclustering}}
Using the notation of Theorem~\ref{thm:sat3}, let $B= B(H_n(\pi,k,m))$ and let $b=b(H_n(\pi,k,m))$. A flippable set in the 2-core of $B$ is a set of columns $S$ in $B_{\text{2-core}}$ such that, for some $x,y$ with $B_{\text{2-core}}x=B_{\text{2-core}}y=b$, we have $x\Delta y=S$.
For any binary matirx $A$, let $m_*(A)\times n_*(A)$ be the dimensions of the 2-core of $A$.

\begin{lemma}\label{lem:deltaseparated} For all constant $\eta>0$ there exists a constant $\delta>0$ for which the following holds:  for all $d\ge d_k^\star+\eta$, a.a.s.\ every minimal flippable set in the 2-core of $B$ is a flippable cycle or has size at least $\delta n$.
\end{lemma}

\proof This lemma follows by exactly the same proof as~\cite[Lemma 51]{achlioptas2015solution}. Although their statement is slightly different, which fixed $d>d_k^\star$ first and then chooses $\delta$ afterwards. However, the same proof actually shows that a universal $\delta$ can be chosen as long as $d-d_k^\star$ is bounded by $\eta$ from below. \qed
\smallskip

%\begin{lemma}[\cite{achlioptas2015solution}]\label{lem:deltaseparated}
 %   For all $d>d^\star_k$, there is a $\delta>0$ such that a.a.s. every minimal flippable set in the 2-core of $B$ is a flippable cycle or has size at least $\delta n$.
%\end{lemma}

By the subsubsequence principle, we may assume that $km/n\to d$ for some constant $d$. 
By the theorem hypothesis $d\ge d_k+k\eps$.
Let $\eta=d_k-d_k^\star$ and let $\delta=\delta(\eta)$ be given by Lemma~\ref{lem:deltaseparated}.  By Lemma~\ref{lem:deltaseparated}, every minimal flippable set in the 2-core of $B$ is a flippable cycle or has size at least $\delta n$.
By Theorem 1.6 in \cite{ayre2020satisfiability}, $n_*(B)/n\to \pi_k(d)$ in probability if $m\sim dn/k$, where
$\pi_k:(0,\infty)\to \RR$ is given by
\begin{equation}
    \pi_k(d)=\rho_{k,d}-d\rho_{k,d}^{k-1}+d\rho_{k,d}^k.
\end{equation}

Note that the function $\pi_k$ is monotonically increasing for $d>d_k^\star$, with $\pi_k(d)<1$ for all $d>d_k^\star$ and $\lim_{d\to\infty} \pi_k(d)=1$. The implicit function theorem shows that $\pi_k$ is continuous on $(d_k^\star,\infty)$.

\begin{claim}\label{claim:2coreinduction}
    Let $i\geq 1$ be fixed. If $d\geq d_k+k\eps$ is such that $\pi_k(d_k)+(i-1)\delta/2\leq\pi_k(d)<\pi_k(d_k)+i\delta /2$, then a.a.s.\ every solution to $Bx=b$ is cycle equivalent to $\sigma^n$ for some $\sigma\in Q_\pi$.
\end{claim}
\begin{proof}[Proof of Claim~\ref{claim:2coreinduction}]
    We prove by induction on $i$. Suppose $i=1$, and suppose $d>0$ is such that $\pi_k(d_k)\le \pi_k(d)< \pi_k(d_k)+\delta/2$. Let $\eps_n$ be chosen to satisfy Lemma~\ref{indeprows}, and let $m'=(d_k/k-\eps_n)n$ and $m''=m-m'$. Let $B'$ be the submatrix consisting of the first $m'$ rows of $B$, and let $B''$ be the submatrix consisting of the last $m''$ rows of $B$. Let $R_n$ and $C_n$ be the rows and columns of $B'_{\text{2-core}}$ respectively, and let $R'_n$ be the set of rows of $B''$ whose entries are zero outside of $C_n$. Let $\Sigma$ be the set of solutions to $B_{\text{2-core}}'\, x=b\mid_{R_n}$. These are exactly the same definitions as in the proof of Theorem~\ref{thm:sat3}. Thus, we immediately have that a.a.s.\ $|R'_n|=\Omega(m'')$ by~\eqref{eq:size-of-R'}, and $|\Sigma|=q^{O(\eps_n n)}$ by Claim~\ref{claim:top-system}.
    
    We call a solution $x$ of $B'_{\text{2-core}}x=b\mid_{R_n}$ non-trivial if for all $\sigma\in Q_{\pi}$,
    $$
    |\{v\in C_n: x_v\neq \sigma\}|\geq \delta n .
    $$
    The following claim is an analog of Claim~\ref{claim:bottom-system}, with nearly identical proof, which we will briefly sketch.
    \begin{claim}\label{claim:bottom}
        Condition on $R_n'$. There is a $\alpha>0$ such that, for all rows $i\in R_n'$ and non-trivial $x$,
    \end{claim}
    \begin{equation*}
        \pr[B\mid _{\{i\}\times C_n}x=b_i]\leq 1-\alpha
    \end{equation*}
    We complete the proof for $i=1$ assuming Claims~\ref{claim:top-system} and \ref{claim:bottom}. Similarly as before, rows in $R'_n$ are independent. By~\eqref{eq:size-of-R'} and Claim~\ref{claim:bottom}, for every non-trivial $x$,
\begin{equation*}
    \pr[B|_{R'_n\times C_n}\, x = b|_{R'_n}]\le (1-\alpha)^{\Omega(m'')} \le \exp(- \Omega(\alpha(\eps+\eps_n)n)).
\end{equation*}

We may assume that $|\Sigma|=q^{O(\eps_n n)}$ by Claim~\ref{claim:top-system}. Then, the probability that there exists a nontrivial $x\in \Sigma$ such that $B|_{R'_n\times C_n}\, x = b|_{R'_n}$ is $q^{O(\eps_n n)} \exp(- \Omega(\alpha(\eps+\eps_n)n))=o(1)$. Therefore a.a.s.\ all solutions $x$ of $Bx=b$ have a trivial restriction $x\mid_{C_n}$ to $C_n$. Therefore, by Lemma \ref{lem:deltaseparated}, all solutions $x$ of $Bx=b$ are cycle equivalent, with respect to the $2$-core of $B'$, to $\sigma^n$ for some $\sigma\in Q_{\pi}$. That is, $x$ and $\sigma^n$ agree on all the variables in the 2-core of $B'$ (instead of $B_{\text{2-core}}$), except on some variables contained in a flippable cycle of $B'_{\text{2-core}}$.

Next we show that a.a.s.\ all solutions of $Bx=b$ are cycle equivalent, with respect to the 2-core of $B$, to $\sigma^n$ for some $\sigma\in Q_{\pi}$. By Lemmas~\ref{lem:cyclebound} and \ref{lem:deltaseparated}, and the assumption that $\pi_k(d)<\pi_k(d_k)+\delta/2$ we may assume that $(B, B')$ satisfies the following properties, which holds with probability $1-o(1)$: 
\begin{itemize}
    \item[(P1)] the number of vertices in flippable cycles of the 2-core is $o(n)$;
    \item[(P2)] every minimal flippable set in the 2-core of $B$ is a flippable cycle or has size at least $\delta n$;
    \item[(P3)] $n_*(B)/n< \pi_k(d_k)+\delta/2$, and that $n_*(B')/n>\pi_k(d_k)-\delta/4$.
\end{itemize}
 Let $\hat x$ satisfy $B\hat x=b$ and suppose for a contradiction that $\hat x$ is not cycle equivalent to any $\sigma^n$, with respect to the 2-core of $B$, where $\sigma\in Q_{\pi}$. Then, by (P2), $\hat x$ differs from each $\sigma^n$, where $\sigma\in Q_{\pi}$, on at least $\delta n$ variables in the 2-core of $B$. Moreover, $n_*(B)/n-n_*(B')/n\le 3\delta n/4$ by (P3), and the variables of the 2-core of $B'$ is a subset of the variables in the 2-core of $B$. It follows then that $\hat x$ must differ from each $\sigma^n$, where $\sigma\in Q_{\pi}$, on at least $\delta n/4$ variables of the 2-core of $B'$ a.a.s. By (P1), the solution $\hat x$ is not cycle equivalent to a constant solution on the 2-core of $B'$. This is a contradiction since we have shown that a.a.s.\ all solutions of $Bx=b$ are cycle equivalent, with respect to the 2-core of $B'$, to a constant solution. Hence, $\hat x$ is cycle equivalent, with respect to $B_{\text{2-core}}$, to $\sigma^n$ for some $\sigma\in Q_{\pi}$.

\begin{proof}[Proof of Claim~\ref{claim:bottom}]
        Let $x\in \Omega^n$ be non-trivial and let $i$ be a row in $R_n'$. Similar to~\eqref{eq:bad}, we show 
        \begin{equation}
        \pr[B\mid_ix\neq b_i]=\Omega(1).\label{eq:bad2}
        \end{equation}
         Let $(v_1,\ldots, v_k)=\partial i$. Let $\sigma\in \Omega$ be such that $|\{v\in C_n: x_v= \sigma\}|\geq |C_n|/r$. We consider two cases.

First we suppose that there is a $\psi\in \supp(\pi)$ such that $\sigma^k$ does not satisfy $\psi$. 
        By the same argument leading to~\eqref{eq:bad}, we immediately have~\eqref{eq:bad2}.
        
        Now suppose that $\sigma^k$ satisfies all $\psi\in \supp(\pi)$. Since $x$ is non-trivial, we have $|\{v\in C_n: x_v\neq \sigma\}|\geq \delta  n$.  Since each $\psi\in \supp(\pi)$ is UE and is satisfied by $\sigma^k$, we know that $(\sigma^{k-1},\sigma')$ does not satisfy $\psi$ for any $\sigma'\neq \sigma$. Therefore,
        \begin{equation*}
            \pr[B\mid_ix\neq b_i]\geq \pr[|\{j\in[k]:x_{v_j}\neq \sigma\}|=1]=\Omega(1)
        \end{equation*}
        where the last eqaulity above holds because $|\{v\in C_n:x_v=\sigma\}|\geq |C_n|/q$ and $|\{v\in C_n: x_v\neq \sigma\}|\geq \delta n$. This verifies~\eqref{eq:bad2}.
    \end{proof}

For the induction step, suppose that $i\ge 2$ and that the claim holds for smaller values of $i$. If $\pi_k(d_k)+(i-1)\delta/2\geq 1$, then there is no $d>0$ that satisfies the hypothesis and the claim holds vacuously. Suppose $\pi_k(d_k)+(i-1)\delta/2<1$ and suppose that $d>0$ is such that $\pi_k(d_k)+(i-1)\delta/2<\pi_k(d)<\pi_k(d_k)+i\delta/2$. Since $i\ge 2$, there is a $d'>d_k$ such that $\pi_k(d_k)+(i-3/2)\delta/2<\pi_k(d')<\pi_k(d_k)+(i-1)\delta/2$. Let $B'$ be the submatrix of $B$ consisting of the first $d'n/k$ rows of $B$ and let $b'$ be the vector consisting of the first $d'n/k$ entries of $b$. By the induction hypothesis, a.a.s.\ every solution to $B'x=b'$ is cycle equivalent, with respect to $B'_{\text{2-core}}$, to $\sigma^n$ for some $\sigma\in Q_{\pi}$.

As in the base case, we may assume that $(B, B')$ has properties (P1), (P2) and
\begin{enumerate}
    \item[(P3')]  $n_*(B)/n< \pi_k(d_k)+i\delta/2$, and that $n_*(B')/n>\pi_k(d_k)-(i-3/2)\delta/2$.
\end{enumerate}

Let $\hat x$ satisfy $B\hat x=b$ and suppose for a contradiction that $\hat x$ is not cycle equivalent, with respect to the 2-core of $B$,  to any $\sigma^n$, where $\sigma\in Q_\pi$ . Then by (P2), $\hat x$ differs from each $\sigma^n$, $\sigma\in Q_\pi$, on at least $\delta n$ variables in the 2-core of $B$. But by (P3'), $\hat x$ must differ from any $\sigma^n$, $\sigma\in Q_\pi$, on at least $\delta n/4$ variables of the 2-core of $B'$ a.a.s. By (P1), the solution $\hat x$ is not cycle equivalent, with respect to the 2-core of $B'$, to any constant solution, which contradicts with the inductive hypothesis. Hence, $\hat x$ is cycle equivalent, with respect to $B_{\text{2-core}}$, to some $\sigma^n$ where $\sigma\in Q_\pi$.

As in part b), the clusters satisfy (b.iii) and (b.iv) by the same proof as~\cite[Theorems 2 and 3]{achlioptas2015solution}. To show that (b.ii) holds, we cannot count the number of solutions to $Bx=b$ as we did in part b) since the rows of $B$ are no longer a.a.s.\ linearly independent. Instead, we prove the following claim, which counts the number of ways that a solution to $B_{\text{2-core}}\,x=b\mid_{R_n}$ can be extended to a solution to $Bx=b$. The proof of Claim~\ref{claim:2-coreext} follows a similar argument to the proof of~\cite[Theorem 2]{achlioptas2015solution}, which we briefly sketch below.

\begin{claim}\label{claim:2-coreext}
    Let $x$ be such that $B_{\text{2-core}}\,x=b\mid _{R_n}$. Then the number of solutions $y$ to $By=b$ with $y\mid_{C_n}=x$ is $r^{n-m-n_*(B)+m_*(B)}$.
\end{claim}
\begin{proof}
    We prove by induction on $n-n_*(B)$. If $n=n_*(B)$, then $B=B_{\text{2-core}}$ and $m=m_*(B)$, so $r^{n-m-n_*(B)+m_*(B)}=1$.

    Suppose $n>n_*(B)$. Then there is a column $v$ of $B$ such that $v$ has at most one nonzero entry. Let $B'$ be the matrix obtained by removing the column $v$ from $B$ along with any row that contains a non-zero entry in column $v$. We then have $B'_{\text{2-core}}=B_{\text{2-core}}$. We consider two cases.

    First suppose that $v$ has no nonzero entries. Then $B'$ is an $m\times (n-1)$ matrix, so by the induction hypothesis there are $r^{(n-1)-m-n_*(B)+m_*(B)}$ solutions to $B'y'=b$ such that $y'\mid_{C_n}=x$. Every such $y'$ can be extended to a solution to $By=b$ by assigning $v$ to any value in $\Omega$, so the number of $y$ such that $By=b$ and $y\mid_{C_n}=x$ is $r^{n-m-n_*(B)+m_*(B)}$.

    No suppose that $v$ has exactly one nonzero entry in row $t$, and let $b'$ be the vector $b$ without the $t$-th entry.  Then $B'$ is an $(m-1)\times (n-1)$ matrix, so by the induction hypothesis there are $r^{n-m-n_*(B)+m_*(B)}$ solutions to $B'y'=b'$ such that $y'\mid_{C_n}=x$. For each such $y'$, since the $t$-th constraint is UE, there is a unique assignment of the variable $v$ that extends $y'$ to a solution $y$ of $By=b$. Therefore there are $r^{n-m-n_*(B)+m_*(B)}$ solutions to $By=b$ such that $y\mid_{C_n}=x$. 
\end{proof}
We now show that (b.ii) holds using Claim~\ref{claim:2-coreext}. By Lemma~\ref{lem:cyclebound}, the number of ways to assign the variable in flippable cycles of $B_{\text{2-core}}$ is $r^{o(n)}$ a.a.s. By Claim~\ref{claim:2-coreext}, each 2-core solution (i.e.\ each solution to $B_{\text{2-core}}\, x= b|_{R_n}$) can be extended to $r^{n-m-n_*(B)+m_*(B)}$ solutions of $Bx=b$. Therefore the number of solutions in a cluster $C_i$ satisfies
\begin{equation*}
    \frac{1}{n}\log |C_i|=\left(1-d/k-\frac{n_*(B)}{n}+\frac{m_*(B)}{n}\right)\log r+o(1) \quad a.a.s.
\end{equation*}
Using~\cite[Theorem 1.6]{ayre2020satisfiability} for $n_*(B)/n$ and $m_*(B)/n$, we get
\begin{equation*}
    \frac{1}{n}\log |C_i|\to(1-d/k-\rho_{k,d}+d\rho_{k,d}^{k-1}-(d-d/k)\rho_{k,d}^k)\log r.
\end{equation*}
Therefore (b.ii) holds.
\end{proof}

    \section{Characterization of group constraint functions for $k\geq 3$}\label{sec:symmetricthm}

    Recall that $\Omega$ is a set of $r$ spins, $k\ge 3$ is an integer, and $\Lambda_{k,r}$ denotes the set of all $k$-ary uniquely extendable constraint functions over $\Omega$. Recall also that for $\psi\in \Lambda_{k,r}$,  $f_\psi:\Omega^{k-1}\to \Omega$ denotes the unique function such that $\sigma f_\psi(\sigma)\models \psi$ for all $\sigma\in \Omega^{k-1}$.

\begin{lemma} \label{psiiUE}
    Let $\Psi\subseteq \Lambda_{k,r}$. The following hold for all $i\geq 0$.
    \begin{enumerate}[(a)]
        \item \label{psiiUE:1} Let $f\in F_i(\Psi)$, and $d\in [(k-1)^i]$. For all $\sigma_1\in \Omega^{d-1}$, $\sigma_2\in \Omega^{(k-1)^i-d}$, the function $\tau \mapsto f(\sigma_1\tau\sigma_2)$, $\tau\in\Omega$ is bijective.
        \item \label{psiiUE:2}Let $\psi\in \Psi$ and let $f_1,\ldots,f_k\in F_i(\Psi)$. Then the constraint $\psi'$ defined by
        \begin{equation*}
            \sigma_1\ldots \sigma_k\models \psi'\iff f_1(\sigma_1)\ldots f_k(\sigma_k)\models \psi,\quad \sigma_1,\ldots,\sigma_k\in \Omega^{(k-1)^i}
        \end{equation*}
        is uniquely extendable.
    \end{enumerate}
\end{lemma}
\begin{proof}
    Since $\psi$ is uniquely extendable, \ref{psiiUE:2} follows from \ref{psiiUE:1}. We prove \ref{psiiUE:1} by induction on $i$. The case for $i=0$ is clear. For $i=1$, let $f\in F_1(\Psi)$.
   By definition, $f=f_{\psi}$ for some $\psi\in \Psi$.
    Let $d\in [k-1]$, let $\sigma_1\in \Omega^{d-1}$, and let $\sigma_2\in \Omega^{k-1-d}$.  For any $\eta\in \Omega$, by unique extendability of $\psi$ there is a unique $\tau \in\Omega$ such that $\sigma_1\tau \sigma_2\eta\models \psi$. Thus there is a unique $\tau$ such that $f(\sigma_1\tau \sigma_2)=f_{\psi}(\sigma_1\tau \sigma_2)=\eta$ for each $\eta\in \Omega$. Therefore  $\tau\mapsto f(\sigma_1\tau \sigma_2)$ is a bijection.

    Now suppose $i\ge 2$ and that the claim holds for all smaller values of $i$. Let $f\in F_i(\Psi)$. Then, by definition,
    \begin{equation*}
        f(\eta_1\ldots \eta_{k-1})=f_{\psi}\bigl (f_1(\eta_1)\ldots f_{k-1}(\eta_{k-1})\bigr ),\quad\eta_1\ldots \eta_{k-1}\in \Omega^{(k-1)^{i-1}},
    \end{equation*}
    for some $\psi\in\Psi$ and some $f_1,\ldots, f_{k-1}\in F_{i-1}(\Psi)$.
    Let $d\in [(k-1)^i]$, let $\sigma_1\in \Omega^{d-1}$, and let $\sigma_2\in \Omega^{(k-1)^i-d}$.

    Let $\pi_{1},\ldots,\pi_{\ell}\in \Omega^{(k-1)^{i-1}}$ and $\pi$ be such that $\pi_1\ldots \pi_{\ell_1}\pi=\sigma_1$ and let $\eta_{1},\ldots,\eta_{\ell_2}\in \Omega^{(k-1)^{i-1}}$ and $\eta$ be such that $\eta\eta_{1}\ldots \eta_{\ell_2}=\sigma_2$. By the induction hypothesis,
    \begin{align}
        &\tau \mapsto f_{\ell_1+1}(\pi\tau \eta),\label{psiiUE:f1}\\
        &\tau \mapsto f_0\bigl (f_1(\pi_{1})\ldots f_{\ell_2}(\pi_{\ell_1})\tau f_{\ell_1+2}(\eta_{1})\ldots f_{k-1}(\eta_{\ell_2})\bigr )\label{psiiUE:f2}
    \end{align}
    are bijective. Since $\tau\mapsto f(\sigma_1\tau\sigma_2)$ is the composition of the two functions (\ref{psiiUE:f1}) and (\ref{psiiUE:f2}), this function is also bijective.
\end{proof}

Recall that a $k$-ary function $f$ on a set $\Omega$ is symmetric if $f(\sigma_1,\ldots,\sigma_k)=f(\sigma_{\pi(1)},\ldots,\sigma_{\pi(k)})$ for all $\pi\in \Pi([k])$ and $\sigma_1,\ldots,\sigma_k\in \Omega$.

\begin{lemma}\label{psiicom}
   Suppose $\Psi\subseteq \Lambda_{k,r}$ is such that every $\psi\in\Psi$ is commutative and every function in $F_2(\Psi)$ is symmetric. Then for all $i\geq 0$, every function in $F_i(\Psi)$ is symmetric. 
\end{lemma}
\begin{proof}
    We prove by induction on $i$. The claim is clear for $i=0,1,2$ by the assumptions. Suppose $i>2$ and that the claim holds for smaller values.

    Let $f\in F_i(\Psi)$. Let $\sigma\in \Omega^{(k-1)^i}$. Write $\sigma$ by $(\sigma_{1,1}\ldots\sigma_{1,k-1})\ldots(\sigma_{k-1,1}\ldots\sigma_{k-1,k-1})$ where each $\sigma_{j,h}\in\Omega^{(k-1)^{i-2}}$. By definition,
    \begin{equation*}
        f(\sigma)=f_{\psi_0}\bigl (f_{\psi_1}\bigl (f_{1,1}(\sigma_{1,1}),\ldots, f_{1,k-1}(\sigma_{1,k-1})\bigr ),\ldots, f_{\psi_{k-1}}\bigl (f_{k-1,1}(\sigma_{k-1,1}),\ldots, f_{k-1,k-1}(\sigma_{k-1,k-1})\bigr )\bigr )
    \end{equation*}
    for some $\psi_0,\psi_1,\ldots,\psi_k\in\Psi$ and for some $f_{1,1},\ldots,f_{1,k-1},\ldots,f_{k-1,k-1}\in F_{i-2}(\Psi)$.
    Let $f'$ be the function defined by
    \begin{align*}
        f'(\sigma_1\ldots \sigma_{k-1})=f_{\psi_0}\bigl (f_{\psi_1}(\sigma_1),\ldots, f_{\psi_{k-1}}(\sigma_{k-1})\bigr ),\quad \sigma_1,\ldots, \sigma_{k-1}\in \Omega^{k-1}.
    \end{align*}
    Since $f'\in F_2(\Psi)$, $f'$ is symmetric. 
    
  Again fix $\sigma\in \Omega^{(k-1)^i}$, and write $\sigma$ by $\sigma_1\ldots\sigma_{k-1}$, where $\sigma_j=\sigma_{j,1}\ldots \sigma_{j,k-1}\text{ for } j\in[k-1]$, where each $\sigma_{j,h}\in\Omega^{(k-1)^{i-2}}$.     To show that $f$ is symmetric, it suffices to show that $f$ is constant by transpositions of any two indices of $\sigma$.

   To specify two indices, let $r_1,s_1,r_2,s_2\in[k-1]$ with $r_1\leq r_2$. Suppose that $\sigma_{r_1,s_1}',\sigma_{r_2,s_2}'\in\Omega^{(k-1)^{i-2}}$ are obtained from $\sigma_{r_1,s_1}\sigma_{r_2,s_2}$ by a transposition of indices. Let $\sigma'$ be obtained from $\sigma$ with $\sigma_{r_1,s_1}$ and $\sigma_{r_2,s_2}$ replaced by $\sigma_{r_1,s_1}'$ and $\sigma_{r_2,s_2}'$ respectively. We write $\sigma'$ in the same block structure as $\sigma$. Therefore, 
    \begin{gather*}
       % \sigma_j=\sigma_{j,1}\ldots \sigma_{j,k-1}\text{ for } j\in[k-1\bigr ),\quad \sigma=\sigma_1\ldots \sigma_{k-1},\\
        \sigma_{r_1}'=\sigma_{r_1,1}\ldots \sigma_{r_1,s_1-1}\sigma_{r_1,s_1}'\sigma_{r_1,s_1+1}\ldots \sigma_{r_1,k-1}, \quad \sigma_{r_2}'=\sigma_{r_2,1}\ldots\sigma_{r_2,s_2-1}\sigma_{r_2,s_2}'\sigma_{r_2,s_2+1} \ldots \sigma_{r_2,k-1}.
     %   \sigma'=\sigma_1\ldots\sigma_{r_1-1}\sigma_{r_1}'\sigma_{r_1+1}\ldots \sigma_{r_2-1}\sigma_{r_2}\sigma_{r_2+1}'\ldots \sigma_{k-1}.
    \end{gather*}
    The following function
    \begin{equation*}
    \eta_1\ldots \eta_{k-1}\mapsto f_{\psi_{r_1}}\bigl (f_{r_1,1}(\eta_1)\ldots f_{r_1,{k-1}}(\eta_{k-1})\bigr ),\quad \eta_1,\ldots, \eta_{k-1}\in \Omega^{(k-1)^{i-2}}
    \end{equation*}
    is symmetric by the inductive hypothesis, so we have $f(\sigma)=f(\sigma')$ if $r_1=r_2$. Suppose $r_1<r_2$. Let $\rho_j=f_{\psi_j}\bigl (f_{j,1}(\sigma_{j,1})\ldots f_{j,k-1}(\sigma_{j,k-1})\bigr )$ for $j\in[k-1]$, and let
    \begin{align*}
        \eta_1&=(\rho_1,\ldots \rho_{r_1-1}),\quad \eta_2=(\rho_{r_1+1},\ldots \rho_{r_2-1}),\quad \eta_3=(\rho_{r_2+1},\ldots, \rho_{k-1}).
    \end{align*}
    Finally, let
    \begin{align*}
        &\pi_{1}^-=f_{r_1,1}(\sigma_{r_1,1})\ldots f_{r_1,s_1-1}(\sigma_{r_1,s_1-1})&\pi_{1}^+=f_{r_1,s_1+1}(\sigma_{r_1,s_1+1})\ldots f_{r_1,k-1}(\sigma_{r_1,k-1})\\
        &\pi_{2}^-=f_{r_2,1}(\sigma_{r_2,1})\ldots f_{r_2,s_2-1}(\sigma_{r_2,s_2-1})&\pi_{2}^+=f_{r_2,s_2+1}(\sigma_{r_2,s_2+1})\ldots f_{r_2,k-1}(\sigma_{r_2,k-1}).
    \end{align*}
    By definition of $f$, we then have
    \begin{equation*}
        f(\sigma)=f_{\psi_0}\bigl (\eta_1, f_{\psi_{r_1}}(\pi_1^-f_{r_1,s_1}(\sigma_{r_1,s_1})\pi^+_1),\eta_2,f_{\psi_{r_2}}(\pi_2^-f_{r_2,s_2}(\sigma_2)\pi^+_2),\eta_3\bigr ).
    \end{equation*}
For each $j=1,2$, rewrite the string $\pi_j^-\pi_j^+$ by $\tau_j \pi_j$ where $\tau_j\in \Omega$, i.e.\ $\tau_j$ is the first spin of $\pi_j^-$.

Since $f_{\psi_{r_1}},f_{\psi_{r_2}}$ are symmetric, we have
    \begin{equation*}
        f(\sigma)=f_{\psi_0}\bigl (\eta_1 f_{\psi_{r_1}}(f_{r_1,s_1}(\sigma_{r_1,s_1})\tau_1\pi_1)\eta_2f_{\psi_{r_2}}(f_{r_2,s_2}(\sigma_{r_2,s_2})\tau_2\pi_2)\eta_3\bigr ).
    \end{equation*}
    and since $f'$ is symmetric,
    \begin{align*}
        &f_{\psi_0}\bigl (\eta_1 f_{\psi_{r_1}}(f_{r_1,s_1}(\sigma_{r_1,s_1})\tau_1\pi_1)\eta_2f_{\psi_{r_2}}(f_{r_2,s_2}(\sigma_{r_2,s_2})\tau_2\pi_2)\eta_3\bigr )\\
        &\hspace{2cm}=f_{\psi_0}\bigl (\eta_1 f_{\psi_{r_1}}(f_{r_1,s_1}(\sigma_{r_1,s_1})f_{r_2,s_2}(\sigma_{r_2,s_2})\pi_1)\eta_2f_{\psi_{r_2}}(\tau_1\tau_2\pi_2)\eta_3\bigr ).
    \end{align*}
    By the induction hypothesis, the function
    \begin{equation*}
        \gamma_1\gamma_2\mapsto f_{\psi_{r_1}}\bigl (f_{r_1,s_1}(\gamma_1)f_{r_1,s_2}(\gamma_2)\pi_1\bigr ),\quad\gamma_1,\gamma_2\in \Omega^{(k-1)^{i-2}}
    \end{equation*}
    is also symmetric. Thus
    \begin{equation*}
        f_{\psi_{r_1}}\bigl (f_{r_1,s_1}(\sigma_{r_1,s_1})f_{r_2,s_2}(\sigma_{r_2,s_2})\pi_1\bigr )=f_{\psi_{r_1}}\bigl (f_{r_1,s_1}(\sigma_{r_1,s_1}')f_{r_2,s_2}(\sigma_{r_2,s_2}')\pi_1\bigr ).
    \end{equation*}
    and so
    \begin{align*}
        &f_{\psi_0}\bigl (\eta_1 f_{\psi_{r_1}}(f_{r_1,s_1}(\sigma_{r_1,s_1})f_{r_2,s_2}(\sigma_{r_2,s_2})\pi_1)\eta_2f_{\psi_{r_2}}(\tau_1\tau_2\pi_2)\eta_3\bigr )\\
        &\hspace{2cm}=f_{\psi_0}\bigl (\eta_1 f_{\psi_{r_1}}\bigl (f_{r_1,s_1}(\sigma_{r_1,s_1}')f_{r_2,s_2}(\sigma_{r_2,s_2}')\pi_1\bigr )\eta_2f_{\psi_{r_2}}(\tau_1\tau_2\pi_2)\eta_3\bigr ).
    \end{align*}
    Applying symmetry of $f'$, $f_{\psi_{r_1}}$, and $f_{\psi_{r_2}}$ again shows that
    \begin{align*}
        &f_{\psi_0}\bigl (\eta_1 f_{\psi_{r_1}}(f_{r_1,s_1}(\sigma_{r_1,s_1}')f_{r_2,s_2}(\sigma_{r_2,s_2}')\pi_1)\eta_2f_{\psi_{r_2}}(\tau_1\tau_2\pi_2)\eta_3\bigr )\\
        &\hspace{2cm}=f_{\psi_0}\bigl (\eta_1 f_{\psi_{r_1}}(f_{r_1,s_1}(\sigma_{r_1,s_1}')\tau_1\pi_1)\eta_2f_{\psi_{r_2}}(f_{r_2,s_2}(\sigma_{r_2,s_2}')\tau_2\pi_2)\eta_3\bigr )\\
        &\hspace{2cm}=f_{\psi_0}\bigl (\eta_1 f_{\psi_{r_1}}(\pi_1^-f_{r_1,s_1}(\sigma_{r_1,s_1}')\pi_2^+)\eta_2f_{\psi_{r_2}}(\pi_2^-f_{r_2,s_2}(\sigma_{r_2,s_2}')\pi_2^+)\eta_3\bigr )\\
        &\hspace{2cm}=f(\sigma')
    \end{align*}
    Therefore $f$ is symmetric.
\end{proof}

Given a multilset $S$ of elements in $\Omega$, let $m_S(x)$ denote the multiplicity of $x$ in $S$.
For two multisets $A$ and $B$ of spins in $\Omega$, define the symmetric difference $A\Delta B$ of them by
\[
m_{A\Delta B}(x)=|m_A(x)-m_B(x)|,\quad \forall x\in \Omega.
\]

\begin{lemma}\label{lem:solndifference}
Suppose $\Psi\subseteq \Lambda_{k,r}$ is such that every $\psi\in\Psi$ is commutative and every function in $F_2(\Psi)$ is symmetric.  Let $\ell\ge 1$ and let $\psi_i\in\Psi$ for $i\in[\ell]$  ($\{\psi_i\}_{i\in [\ell]}$ are not necessarily distinct). Let $\pi_i,\pi_i'\in \Omega^{k}$ for $i\in[\ell]$ be such that $\pi_i\models \psi_i$ and $\pi_i'\models \psi_i$.  Then the multisets of spins in $\pi_1\ldots \pi_\ell$ and $\pi_1'\ldots \pi_\ell'$ do not have a symmetric difference of size two.
\end{lemma}

\begin{proof}
    Let $i\geq 0$ be such that $(k-1)^i\geq \ell$ and let $g\in F_i(\Psi)$. Fix $\psi_0\in \Psi$ and define a pair of functions $f',f'':\Omega^{(k-1)^{i+1}}\to \Omega$ by
    \begin{align*}
        f'(\sigma_1\ldots \sigma_{(k-1)^i})&=g\bigl (f_{\psi_1}(\sigma_1),\ldots, f_{\psi_{\ell}}(\sigma_\ell),f_{\psi_0}(\sigma_{\ell+1}),\ldots, f_{\psi_0}(\sigma_{(k-1)^i})\bigr ),&&\sigma_1,\ldots, \sigma_{(k-1)^i}\in \Omega^{k-1}\\
        f''(\sigma_1\ldots\sigma_{k-1})&=f_{\psi_0}\bigl (g(\sigma_1),\ldots, g(\sigma_{k-1})\bigr ),&&\sigma_1,\ldots, \sigma_{k-1}\in \Omega^{(k-1)^i},
    \end{align*}
    and define $f:\Omega^{(k-1)^{i+2}}\to \Omega$ by 
    \begin{align*}
        f(\sigma_1,\ldots, \sigma_{k-1})=f_{\psi_0}\bigl (f'(\sigma_1)\ldots f'(\sigma_{k-2})f''(\sigma_{k-1})\bigr ),&&\sigma_1,\ldots, \sigma_{k-1}\in \Omega^{(k-1)^{i+1}}.
    \end{align*}
    Then $f',f''\in F_{i+1}(\Psi)$ and $f\in F_{i+2}(\Psi)$, so Lemma \ref{psiicom} shows that $f'$, $f''$, and $f$ are symmetric.
    \begin{claim}\label{claim:solndifference1}
        There is a $\beta \in \Omega$ and a $\sigma\in \Omega^{(k-1)^{i+2}-\ell k}$ such that 
        $f(\hat\pi_1\ldots \hat\pi_\ell \sigma)=\beta$ for all $(\hat\pi_i)_{ i\in[\ell]}$, such that $\hat\pi_i\models \psi_i$ for all $i
        \in [\ell]$.
    \end{claim}
    \begin{proof}
    Fix $\alpha\in \Omega$ and let $\beta=f_{\psi_0}(\alpha^{k-1})$. By Lemma \ref{psiiUE}, there are $\gamma\in \Omega^{(k-1)^i}$ and $\rho\in \Omega^{(k-1)^{i+1}}$ such that $g(\gamma)=f'(\rho)=\beta$. Let $\hat\pi_i\in \Omega^k$ for $i\in[\ell]$ be such that $\hat\pi_i\models \psi_i$. Let $\eta_i\in \Omega^{k-1}$ and $\tau_i\in\Omega$ for $i\in[\ell]$ be such that $\hat\pi_i=\eta_i\tau_i$. We then have $f_{\psi_i}(\eta_i)=\tau_i$, and thus,
    \begin{align*}
        f'(\eta_1\ldots \eta_{\ell} \alpha^{(k-1)^{i+1}-\ell(k-1)})&=g\bigl (f_{\psi_1}(\eta_1)\ldots f_{\psi_\ell}(\sigma_\ell)f_{\psi_0}(\alpha^{k-1})\ldots f_{\psi_0}(\alpha^{k-1})\bigr )\\
        &=g\bigl (\tau_1\ldots \tau_\ell\beta^{(k-1)^i-\ell}\bigr )
    \end{align*}
    and
    \begin{align*}
        f''(\gamma^{k-2}\tau_1\ldots \tau_\ell\beta^{(k-1)^i-\ell})&=f_{\psi_0}\bigl (g(\gamma),\ldots, g(\gamma),g(\tau_1\ldots \tau_\ell\beta^{(k-1)^i-\ell})\bigr )\\
        &=f_{\psi_0}\bigl (\beta^{k-2}g(\tau_1\ldots \tau_\ell\beta^{(k-1)^i-\ell})\bigr ).
    \end{align*}
    Now, by symmetry of $f$,
    \begin{align*}
        &f(\hat\pi_1\ldots \hat\pi_\ell\alpha^{(k-1)^{i+1}-\ell(k-1)}\rho^{k-3}\gamma^{k-2}\beta^{(k-1)^i-\ell})\\
        &\hspace{2cm}=f(\eta_1\ldots \eta_\ell\alpha^{(k-1)^{i+1}-\ell(k-1)}\rho^{k-3}\gamma^{k-2}\tau_1\ldots \tau_\ell\beta^{(k-1)^i-\ell})\\
        &\hspace{2cm}=f_{\psi_0}\bigl (f'(\eta_1\ldots \eta_\ell\alpha^{(k-1)^{i+1}-\ell(k-1)})f'(\rho)\ldots f'(\rho)f''(\gamma^{k-2}\tau_1\ldots \tau_\ell\beta^{(k-1)^i-\ell})\bigr )\\
        &\hspace{2cm}=f_{\psi_0}\bigl (g(\tau_1\ldots \tau_\ell\beta^{(k-1)^i-\ell})\beta^{k-3}f_{\psi_0}(\beta^{k-2}g(\tau_1\ldots \tau_\ell\beta^{(k-1)^i-\ell}))\bigr )\\
        &\hspace{2cm}=f_{\psi_0}\bigl (\beta^{k-3}g(\tau_1\ldots \tau_\ell\beta^{(k-1)^i-\ell})f_{\psi_0}(\beta^{k-2}g(\tau_1\ldots \tau_\ell\beta^{(k-1)^i-\ell}))\bigr ).
    \end{align*}
    By definition for $f_{\psi_0}$, we have
    \begin{align*}
        \beta^{k-2}g(\tau_1\ldots \tau_\ell\beta^{(k-1)^i-\ell})f_{\psi_0}(\beta^{k-2}g(\tau_1\ldots \tau_\ell\beta^{(k-1)^i-\ell}))\models \psi_0,
    \end{align*}
    and so,
    \begin{equation*}
        f_{\psi_0}\bigl (\beta^{k-3}g(\tau_1\ldots \tau_\ell\beta^{(k-1)^i-\ell})f_{\psi_0}(\beta^{k-2}g(\tau_1\ldots \tau_\ell\beta^{(k-1)^i-\ell}))\bigr )=\beta.
    \end{equation*}
    Therefore the claim holds with $\sigma=\alpha^{(k-1)^{i+1}-\ell(k-1)}\rho^{k-3}\gamma^{k-2}\beta^{(k-1)^i-\ell}$.
    \end{proof}
    Now suppose for a contradiction that there are 
    $(\pi_i,\pi_i')_{i\in [\ell]}$ such that $\pi_i,\pi_i'\models \psi_i$ for each $i\in[\ell]$ and such that the multisets of spins in $\pi_1\ldots \pi_\ell$ and $\pi_1'\ldots \pi_\ell'$ have symmetric difference of size two. Let $\pi\in \Omega^{\ell k-1}$ and $\tau,\tau'\in \Omega$ be such that $\tau\pi$ is a permutation of $\pi_1\ldots \pi_\ell$ and $\tau'\pi$ is a permutation of $\pi_1'\ldots \pi_\ell'$. Then, by symmetry of $f$ and Claim \ref{claim:solndifference1}, there exists $\sigma\in \Omega^{(k-1)^{i+2}-k\ell}$ such that
    \begin{equation*}
    f(\tau \pi\sigma)=f(\pi_1\ldots\pi_\ell\sigma)=f(\pi_1'\ldots \pi_\ell'\sigma)=f(\tau'\pi\sigma).
    \end{equation*}
    But this contradicts the injectivity of the function
    \begin{equation*}
        \tau \mapsto f(\tau\pi\sigma)
    \end{equation*}
    from Lemma \ref{psiiUE}.
    
    \end{proof}

\begin{proof}[Proof of Theorem~\ref{symthm}]
    Suppose \ref{symthm1} holds. Let $(\Omega,+)$ be a group and let $\{b_\psi\}\subset \Omega$ be a set of spins such that $\psi=\psi_{(\Omega,+),b_\psi}$ for all $\psi\in \Psi$. Let $e\in \Omega$ be the identity in $(\Omega,+)$. For a string $\sigma=(\sigma_1,\ldots, \sigma_\ell)\in \Omega^\ell$, where $\ell\ge 1$, denote
    \begin{equation*}
        \sigma^\oplus=\sum_{i=1}^\ell \sigma_i.
    \end{equation*}
    Then, for any $\psi\in \Psi$ and $g_1,g_2\in \Omega$, commutativity of $\psi$ gives
    \begin{equation*}
        g_1g_2e^{k-3}f_\psi(g_1g_2e^{k-3})\models \psi,\quad g_2g_1e^{k-3}f_\psi(g_1g_2e^{k-3})\models \psi
    \end{equation*}
    and therefore
    \begin{equation*}
        g_1+g_2+f_\psi(g_1g_2e^{k-3})=b_{\psi}=g_2+g_1+f_\psi(g_1g_2e^{k-3})
    \end{equation*}
    Hence $g_1+g_2=g_2+g_1$ for all $g_1,g_2\in \Omega$, and so $(\Omega,+)$ is an Abelian group.

    Now let $f\in F_2(\Psi)$. Let $\psi_0,\ldots, \psi_{k-1}\in \Psi$ be such that
    \begin{align*}
        f(\sigma_1\ldots \sigma_{k-1})=f_{\psi_0}\bigl (f_{\psi_1}(\sigma_1)\ldots f_{\psi_{k-1}}(\sigma_{k-1})\bigr ),&&\sigma_1,\ldots, \sigma_{k-1}\in \Omega^{k-1}.
    \end{align*}
    For any $i\in0,\ldots, k-1$, we have $\sigma f_{\psi_i}(\sigma)\models \psi$ for any $\sigma\in \Omega^{k-1}$, which yields
    \begin{equation*}
        \sigma^\oplus +f_{\psi_i}(\sigma)=b_{\psi_i}
    \end{equation*}
    and hence, for any $\sigma_1,\ldots, \sigma_{k-1}\in \Omega^{k-1}$,
    \begin{align*}
        b_{\psi_0}&= \sum_{i=1}^{k-1}f_{\psi_i}(\sigma_i) + f_{\psi_0}\bigl (f_{\psi_1}(\sigma_1)\ldots f_{\psi_{k-1}}(\sigma_{k-1})\bigr )\\
        &= \sum_{i=1}^{k-1}b_{\psi_i}-\sigma_i^\oplus  + f_{\psi_0}\bigl (f_{\psi_1}(\sigma_1)\ldots f_{\psi_{k-1}}(\sigma_{k-1})\bigr )\\
    \end{align*}
    which gives
    \begin{align*}
        f(\sigma_1\ldots \sigma_k)= b_{\psi_0}-\sum_{i=1}^{k-1}b_{\psi_i}+\sum_{i=1}^{k-1}\sigma_i^\oplus.
    \end{align*}
    Since $(\Omega,+)$ is Abelian, this shows that $f$ is symmetric.\\
    
    For the converse direction, suppose \ref{symthm2} holds. Fix $\alpha\in \Omega$ and let $\beta_\psi=f_\psi(\alpha\ldots \alpha)$ for all $\psi\in \Psi$. Let $G'={\mathbb Z}^{(\Omega)}$ be the free Abelian group with basis $\Omega$, so that the elements of $G'$ are formal $\ZZ$-linear combinations of elements of $\Omega$; that is,
    \begin{equation*}
    G'=\left\{\sum_{\tau\in \Omega} n_\tau\tau:\ (n_{\tau})_{\tau\in \Omega}\in {\mathbb Z}^{\Omega}\right\}.
    \end{equation*}

Let $N$ be the subgroup of $G'$ generated by the set
\begin{equation*}
\big\{\sigma^\oplus_1-\sigma_2^\oplus: \sigma_1,\sigma_2\models \psi,\psi\in\Psi\big\}\cup \big\{\alpha\big\}.
\end{equation*}
Notice that every element in $N$ can be expressed in the form of
    \begin{equation}
   a\alpha + \sum_{i=1}^\ell \sigma_{i,1}^\oplus-\sigma_{i,2}^\oplus, \quad \sigma_{i,1},\sigma_{i,2}\models \psi_i\text{ for some }\psi_i\in\Psi, \ \ell\ge 0,\ a\in {\mathbb Z}. \label{eq:elements-in-N}
    \end{equation}

Since $G'$ is Abelian, $N$ is a normal subgroup. Let $(G,+)$ be the quotient group $G'/N$ and let $\bar a$ denote the projection of $a\in G'$ in $G$. Let $\pi:\Omega\to G$ be the restriction of the projection map to $\Omega$, where $\Omega$ is considered as a subset of $G'$. Let $b_\psi=\bar\beta_\psi$ for all $\psi\in \Psi$. The converse direction of the theorem is implied by the following two claims.

\begin{claim}\label{claim:G}
    The restriction $\pi$ is a bijection.
\end{claim}

\begin{claim}\label{claim:equivalence}
    For any $\tau_1,\ldots, \tau_k\in \Omega$ and $\psi\in \Psi$, we have $\tau_1\ldots \tau_k\models \psi$ if and only if $\sum_{i=1}^k\pi(\tau_i)=b_{\psi}$.
\end{claim}

\noindent {\em Proof of Claim~\ref{claim:G}.\ } Let $\bar\Omega\subset G$ be the image of $\Omega$ in $G$ under $\pi$. To show that $G=\bar\Omega$, we start by showing that $\bar\Omega$ is closed under taking inverses.
    Let $\tau \in \Omega$ and fix $\psi\in \Psi$. By unique extendability, there is a unique $\tau'\in \Omega$ such that $\tau\tau'\alpha\ldots\alpha\beta_\psi\models \psi$. Since $\alpha\ldots \alpha \beta_\psi\models \psi$, we have
    \begin{equation*}
        (\tau+\tau'+(k-3)\alpha+\beta_\psi)-((k-1)\alpha+\beta_\psi)=\tau+\tau'-2\alpha\in N.
    \end{equation*}
    Since $\alpha\in N$, we get $\tau+\tau'\in N$, and thus $-\bar \tau=\overline{\tau'}\in \bar \Omega$.
    
 Therefore, for any $\bar \sigma\in G$, we can write $\bar\sigma=\sum_{i=1}^\ell \bar \tau_i$ for some $\tau_1,\ldots,\tau_\ell\in \Omega$, since any $-\tau_i$ can be expressed as $\overline{\tau'}$ for some $\tau'\in \Omega$. To show that $G\subset \bar \Omega$ we will show that $\sum_{i=1}^\ell \bar\tau_i\in\bar\Omega$ for all $\tau_1,\ldots, \tau_\ell\in \Omega$ by induction on $\ell$. 
 
 The case for $\ell=1$ is clear. Suppose $\ell>1$, let $\tau_1,\ldots,\tau_\ell\in \Omega$, and fix $\psi\in \Psi$. By unique extendability of $\psi$ there is a unique $\tau\in\Omega$ such that $\tau_1\tau_2\tau\alpha\ldots\alpha\models \psi$, and thus $\bar\tau_1+\bar\tau_2+\bar\tau=\bar\beta_\psi$. There is also a unique $\tau'\in\Omega$ such that $\tau'\tau\alpha\dots\alpha\models \psi$, and thus $\overline{\tau'}+\bar\tau=\bar\beta_\psi$. Therefore $\bar\tau_1+\bar\tau_2=\overline{\tau'}$. We then get $\sum_{i=1}^\ell\bar\tau_i=\bar\tau'+\sum_{i=3}^\ell\bar \tau_i$, with $\bar\tau'+\sum_{i=3}^\ell\bar \tau_i\in \bar\Omega$ by the induction hypothesis. Hence $G=\bar\Omega$ and hence $\pi$ is surjective.

    To show that $\pi$ is injective, suppose $\tau,\tau'\in \Omega$ are such that $\bar \tau=\overline{\tau'}$, and suppose for a contradiction that $\tau\neq \tau'$. By~(\ref{eq:elements-in-N}),
    \begin{equation}\label{eq:spins}
    \tau'=\tau + a\alpha+ \sum_{i=1}^\ell \sigma_{i,1}^\oplus-\sigma_{i,2}^\oplus
    \end{equation}
    for some $a\in \ZZ$, $\sigma_{i,1},\sigma_{i,2}\in \Omega^k$, $\ell\geq 0$ such that, for each $i\in[\ell]$, $\sigma_{i,1},\sigma_{i,k}\models \psi_i$ for some $\psi_i\in \Psi$. Since the number of spins, counting multiplicities, on the right side of (\ref{eq:spins}) is $1$, we have $a=0$. Since $\tau\neq \tau'$, (\ref{eq:spins}) then implies that the symmetric difference between the multisets of spins in $\sigma_{1,1}\ldots \sigma_{\ell,1}$ and $\sigma_{1,2}\ldots \sigma_{\ell,2}$ has size two, which contradicts Lemma \ref{lem:solndifference}. Hence $\tau=\tau'$, as desired. \qed \smallskip

   \noindent {\em Proof of Claim~\ref{claim:equivalence}.\ } Let $\psi\in \Psi$. For each solution $\tau_1\ldots \tau_k$ of $\psi$, define $\tilde\pi(\tau_1\ldots \tau_k)$ by
   \begin{equation*}
       \tilde\pi(\tau_1\ldots \tau_k)=\pi(\tau_1)\ldots \pi(\tau_k)\in G^k.
   \end{equation*}
   To prove Claim \ref{claim:equivalence}, we will show that $\tilde \pi$ is a bijection between the solution sets of $\psi$ and $\psi_{G,b_\psi}$. Suppose $\tau_1\ldots\tau_k\in\Omega^k$ satisfies $\psi$. Then $\bar\tau_1+\dots+\bar\tau_k=\bar \beta_\psi$ by definition of $N$, and so $\tilde \pi(\tau_1\ldots \tau_k)=\bar \tau_1\ldots \bar \tau_k$ satisfies $\psi_{G,b_\psi}$. Therefore $\tilde \pi$ is function from the solution set of $\psi$ to the solution set of $\psi_{G,b_\psi}$. By Claim \ref{claim:G}, $\pi$ is injective, so $\tilde\pi$ is also injective. Moreover, Claim \ref{claim:G} shows that $|G|=|\Omega|$. Since $\psi$ and $\psi_{G,b_\psi}$ are both UE, it follows that $\psi$ and $\psi_{G,b_\psi}$ each have exactly $|\Omega|^{k-1}$ solutions. Therefore $\tilde \pi$ is also surjective, as desired. \qed \smallskip

   By Claim \ref{claim:G}, we can define a group operation $(\Omega, +)$ by
   \begin{equation*}
       \tau_1+\tau_2=\pi^{-1}(\pi(\tau_1)+\pi(\tau_2)),\quad \tau_1,\tau_2\in \Omega.
   \end{equation*}
   By Claim \ref{claim:equivalence}, we have $\psi=\psi_{(\Omega,+),\beta_\psi}$ for all $\psi\in \Psi$ for this choice of $(\Omega,+)$.
\end{proof}

\begin{proof}[Proof of Theorem~\ref{cor:reducible}]
    By Theorem \ref{symthm}, it suffices to show that $\Psi$ is reducible if and only every function in $F_2(\Psi)$ is symmetric. Suppose $\Psi$ is reducible. Since the functions $f_\psi$, $\psi\in \Psi$ are symmetric, it suffices to show that 
    \begin{equation*}
        f_{\psi_0} \bigl (f_{\psi_1}(\tau_1\sigma_1)f_{\psi_2}(\tau_2\sigma_2)f_{\psi_3}(\sigma_3)\ldots f_{\psi_{k-1}}(\sigma_{k-1})\bigr )=f_{\psi_0} \bigl (f_{\psi_1}(\tau_2\sigma_1)f_{\psi_2}(\tau_1\sigma_2)f_{\psi_3}(\sigma_3)\ldots f_{\psi_{k-1}}(\sigma_{k-1})\bigr )
    \end{equation*}
    for all $\tau_1,\tau_2\in \Omega$, $\sigma_1,\sigma_2\in \Omega^{k-2}$, $\sigma_3,\ldots,\sigma_{k-1}\in \Omega^{k-1}$, and $\psi_0,\ldots, \psi_{k-1}\in \Psi$. Let $\tau_1,\tau_2$, $\sigma_1,\sigma_2,\sigma_3,\ldots,\sigma_{k-1}$, $\psi_0,\ldots, \psi_{k-1}$ be given. For notational convenience, let 
    \begin{align}
    \tau&=f_{\psi_3}(\sigma_3)\ldots f_{\psi_{k-1}}(\sigma_{k-1})\nonumber\\
    z_1&=f_{\psi_0} \bigl (f_{\psi_1}(\tau_1\sigma_1)f_{\psi_2}(\tau_2\sigma_2)\tau\bigr )\label{z1}\\ z_2&=f_{\psi_0} \bigl (f_{\psi_1}(\tau_2\sigma_1)f_{\psi_2}(\tau_1\sigma_2)\tau\bigr )\label{z2}
    \end{align}
    and we aim to show that $z_1=z_2.$ Observe that if there is an $\eta\in \Omega^{k-2}$ such that $f_{\psi_1}(\tau_1\sigma_1)f_{\psi_2}(\tau_2\sigma_2)\eta\models\psi_0$ and $f_{\psi_1}(\tau_2\sigma_1)f_{\psi_2}(\tau_1\sigma_2)\eta\models\psi_0$, then by the reducibility of $\Psi$ and~\eqref{z1} it follows that
    \begin{equation}
    f_{\psi_1}(\tau_2\sigma_1)f_{\psi_2}(\tau_1\sigma_2)\tau z_1 \models \psi_0, \label{eq:swap1}
    \end{equation}
    and then $z_1=z_2$ follows by~\eqref{eq:swap1},~\eqref{z2} and the unique extendability of $\psi_0$.
    Therefore it suffices to prove the following claim.
    \begin{claim}\label{claim:commonextension}
        For all $\psi_0,\psi_1,\psi_2\in \Psi$, $\tau_1,\tau_2\in \Omega$, and $\sigma_1,\sigma_2\in \Omega^{k-2}$, there is an $\eta\in \Omega^{k-2}$ such that $f_{\psi_1}(\tau_1\sigma_1)f_{\psi_2}(\tau_2\sigma_2)\eta\models \psi_0$ and $f_{\psi_1}(\tau_2\sigma_1)f_{\psi_2}(\tau_1\sigma_2)\eta\models\psi_0$.
    \end{claim}
    \noindent {\em Proof of Claim~\ref{claim:commonextension}.\ }
        Since $k\geq 4$, there is an $\eta'\in \Omega^{k-3}$ such that 
        \begin{equation}\label{eq:redclaim1}
            f_{\psi_1}(\tau_1\sigma_1)f_{\psi_2}(\tau_2\sigma_2)\tau_2\eta'\models\psi_0.
        \end{equation}
        We also have
        \begin{equation*}
            f_{\psi_2}(\tau_2\sigma_2)\tau_2\sigma_2\models\psi_2,\quad f_{\psi_2}(\tau_1\sigma_2)\tau_1\sigma_2\models\psi_2.
        \end{equation*}
Notice from the above that $f_{\psi_2}(\tau_2\sigma_2)\tau_2$ and $f_{\psi_2}(\tau_1\sigma_2)\tau_1$ have the same extension $\sigma_2$ to satisfy $\psi_2$, and therefore, by reducibility and Remark~\ref{r:extension-exchange},  we can replace $f_{\psi_2}(\tau_2\sigma_2)\tau_2$ in (\ref{eq:redclaim1}) by  $f_{\psi_2}(\tau_1\sigma_2)\tau_1$ to get 
        \[
            f_{\psi_1}(\tau_1\sigma_1)f_{\psi_2}(\tau_1\sigma_2)\tau_1\eta'\models\psi_0,
        \]
which yields the following by swapping the positions of the spins in the string:
\begin{equation}\label{eq:redclaim2}
f_{\psi_1}(\tau_1\sigma_1)\tau_1f_{\psi_2}(\tau_1\sigma_2)\eta'\models\psi_0.
\end{equation}
        
        Similarly, we have
        \begin{equation*}
            f_{\psi_1}(\tau_1\sigma_1)\tau_1\sigma_1\models{\psi_1},\quad f_{\psi_1}(\tau_2\sigma_1)\tau_2\sigma_1\models{\psi_1},
        \end{equation*}
and        so we can use reducibility and Remark~\ref{r:extension-exchange} again to replace $f_{\psi_1}(\tau_1\sigma_1)\tau_1$ by $f_{\psi_1}(\tau_2\sigma_1)\tau_2$ in (\ref{eq:redclaim2}) to get 
        \begin{equation*}
            f_{\psi_1}(\tau_2\sigma_1)\tau_2f_{\psi_2}(\tau_1\sigma_2)\eta'\models\psi_0,
        \end{equation*}
        which yields
        \begin{equation}
        f_{\psi_1}(\tau_2\sigma_1)f_{\psi_2}(\tau_1\sigma_2)\tau_2\eta'\models\psi_0.\label{eq:lastline}
        \end{equation}
        Therefore, by~\eqref{eq:redclaim1} and~\eqref{eq:lastline}, $\eta=\tau_2\eta'$ satisfies the claim.
    \qed\smallskip

    For the converse direction, suppose every function in $F_2(\Psi)$ is symmetric. Let $\psi_1,\psi_2\in \Psi$. Let $d\in \{0,\ldots, k-1\}$ and let $\sigma_1,\sigma_2\in\Omega^d$, $\eta_1,\eta_2\in \Omega^{k-d}$ be such that $\sigma_1\eta_1,\sigma_2\eta_1\models \psi_1$ and $\sigma_1\eta_2\models \psi_2$. Let $\eta_1',\eta_2'\in \Omega^{k-d-1}$ be the prefixes of $\eta_1,\eta_2$ respectively. Since $\sigma_1\eta_1,\sigma_2\eta_1\models \psi_1$, we have $\sigma_1\eta_1'f_{\psi_1}(\sigma_1\eta_1')=\sigma_1\eta_1$ and $\sigma_2\eta_1'f_{\psi_1}(\sigma_2\eta_1')=\sigma_2\eta_1$, and in particular $f_{\psi_1}(\sigma_1\eta_1')=f_{\psi_1}(\sigma_2\eta_1')$. Now, since every function in $F_2(\Psi)$ is symmetric, for any $\eta_3,\ldots \eta_{k-1}\in \Omega^{k-1}$,
    \begin{equation*}
        f_{\psi_1}\bigl (f_{\psi_1}(\sigma_1\eta_1')f_{\psi_2}(\sigma_2\eta_2')f_{\psi_1}(\eta_3)\ldots f_{\psi_{1}}(\eta_{k-1})\bigr )=f_{\psi_1}\bigl (f_{\psi_1}(\sigma_2\eta_1')f_{\psi_2}(\sigma_1\eta_2')f_{\psi_1}(\eta_3)\ldots f_{\psi_{1}}(\eta_{k-1})\bigr )
    \end{equation*}
    and substituting $f_{\psi_1}(\sigma_1\eta_1')$ for $f_{\psi_1}(\sigma_2\eta_1')$ in the right hand side gives
    \begin{equation*}
        f_{\psi_1}\bigl (f_{\psi_1}(\sigma_1\eta_1')f_{\psi_2}(\sigma_2\eta_2')f_{\psi_1}(\eta_3)\ldots f_{\psi_{1}}(\eta_{k-1})\bigr )=f_{\psi_1}\bigl (f_{\psi_1}(\sigma_1\eta_1')f_{\psi_2}(\sigma_1\eta_2')f_{\psi_1}(\eta_3)\ldots f_{\psi_{1}}(\eta_{k-1})\bigr ).
    \end{equation*}
    Thus, by unique extendability of $\psi_1$, we get $f_{\psi_2}(\sigma_2\eta_2')=f_{\psi_2}(\sigma_1\eta_2')$. Since $\sigma_1\eta_2\models \psi_2$, we also have $\sigma_1\eta_2'f_{\psi_{2}}(\sigma_1\eta_2')=\sigma_1\eta_2$, and so $\eta_2=\eta_2'f_{\psi_{2}}(\sigma_1\eta_2')=\eta_2'f_{\psi_{2}}(\sigma_2\eta_2')$. Therefore $\sigma_2\eta_2=\sigma_2\eta_2'f_{\psi_{2}}(\sigma_2\eta_2')\models \psi_2$.
\end{proof}

\section{Non-group UE constraint functions}\label{sec:rediciblethm}
Theorem~\ref{symthm} fully charaterises constraint functions that are equivalent to group constraint functions. In this section, for each $k\ge 3$, we show that not all commutative constraint functions in $\Lambda_{k,r}$ are equivalent to group constraint functions. We give a concrete construction of a family of $k$-ary commutative UE constriant functions $\psi$ such that $F_2(\psi)$ is not symmetric, and consequently $\psi$ is not equivalent to any group constraint function. In particular, these constraint functions are  reducible but not symmetric when $k=3$ and are not reducible when $k\ge 4$.

Our construction relies on an operation that permutes the spins, given as follows. 
Let $\Omega_1$ and $\Omega_2$ be two sets of spins. A map $\theta: \Omega_1^k\mapsto \Pi(\Omega_2)$ is symmetric, if $\theta_{(\tau_1,\dots, \tau_k)}=\theta_{(\tau_{\pi(1)},\dots, \tau_{\pi(k)})}$ for all spins $\tau_1,\dots, \tau_k\in \Omega_1$ and permutations $\pi\in \Pi([k])$. 

\begin{definition}
    Let $\psi_1,\psi_2$ be $k$-ary UE commutative constraints on spin sets $\Omega_1$ and $\Omega_2$ respectively. Let $\theta: \Omega_1^k\mapsto \Pi(\Omega_2)$ be symmetric. Define a $k$-ary constraint $\psi_1\times_\theta \psi_2$ with spin set $\Omega_1\times \Omega_2$ by setting, for all $\alpha=(\alpha_1,\dots, \alpha_k)\in \Omega_1^k$ and $(\beta_1,\dots, \beta_k)\in\Omega_2^k$,
    \begin{equation*}
        ((\alpha_1,\beta_1),\dots,(\alpha_k,\beta_k))\models \psi_1\times_\theta \psi_2\iff (\alpha_1,\dots,\alpha_k)\models \psi_1\text{ and }(\theta_\alpha(\beta_1),\dots, \theta_\alpha(\beta_k))\models\psi_2.
    \end{equation*}
\end{definition}
\begin{lemma}\label{prodUEcom}
    Let $\psi_1,\psi_2$ be $k$-ary UE commutative constraints on spin sets $\Omega_1$ and $\Omega_2$, and let $\theta: \Omega_1^k\to \Pi(\Omega_2)$ be symmetric. Then $\psi_1\times_\theta \psi_2$ is a $k$-ary UE commutative constraint.
\end{lemma}
\begin{proof}
    Let $\alpha=(\alpha_1,\dots, \alpha_k)\in \Omega_1^k$,  $\beta=(\beta_1,\dots, \beta_k)\in\Omega_2^k$, and  $\pi\in\Pi([k])$. Let $\alpha'=(\alpha_{\pi(1)},\dots,\alpha_{\pi(k)})$ and $\beta'=(\beta_{\pi(1)},\dots, \beta_{\pi(k)})$. Since $\theta$ is symmetric, $\theta(\alpha)=\theta(\alpha')$. Thus, by the commutativity of $\psi_1,\psi_2$, we get
    \begin{align*}
        ((\alpha_1,\beta_1),\dots,(\alpha_k,\beta_k))\models \psi_1\times_\theta \psi_2&\iff \alpha\models \psi_1\text{ and }(\theta_\alpha(\beta_1),\dots, \theta_\alpha(\beta_k))\models\psi_2,\\
        &\iff \alpha'\models \psi_1\text{ and }(\theta_{\alpha'}(\beta_{\pi(1)}),\dots, \theta_{\alpha'}(\beta_{\pi(k)}))\models\psi_2,\\
        &\iff ((\alpha_{\pi(1)},\beta_{\pi(1)}),\dots,(\alpha_{\pi(k)},\beta_{\pi(k)}))\models \psi_1\times_\theta \psi_2.
    \end{align*}
    Therefore $\psi_1\times_\theta\psi_2$ is commutative.
    
    To show that $\psi_1\times_\theta\psi_2$ is UE, by commutativity it suffices to fix an assignment of the first $k-1$ arguments of $\psi_1\times_\theta\psi_2$, and show that the last argument must be uniquely determined in order to satisfy $\psi_1\times_\theta\psi_2$. Let $(\alpha_1,\beta_1),\dots, (\alpha_{k-1},\beta_{k-1})\in \Omega_1\times\Omega_2$. Since $\psi_1$ is UE, there is a unique $\alpha_k\in \Omega_1$ such that $\alpha=(\alpha_1,\dots, \alpha_k)\models \psi_1$. Since $\psi_2$ is UE, there is a unique $\beta'_{k}\in \Omega_2$ such that $(\theta_{\alpha}(\beta_1),\dots, \theta_{\alpha}(\beta_{k-1}),\beta'_k)\models \psi_2$. Since $\theta_{\alpha}$ is a permutation, $\beta=\theta_{\alpha}^{-1}(\beta'_k)$ is the unique spin in $\Omega_2$ such that $(\theta_{\alpha}(\beta_1),\dots, \theta_{\alpha}(\beta_{k}))\models \psi_2$. Therefore $(\alpha_k,\beta_k)$ is the unique spin in $\Omega_1\times \Omega_2$ such that $((\alpha_1,\beta_1),\dots, (\alpha_k,\beta_k))\models \psi_1\times_\theta \psi_2$.
\end{proof}
\begin{theorem}
    Suppose $k\geq 3$ and suppose $r=r_1r_2$, where $r_1\geq 3$. Suppose $r_2\geq 3$, or that $r_2=2$ and $k$ is odd. Then there is a $k$-ary commutative UE constraint on a spin set of size $r$ that is not equivalent to any group constraint.
\end{theorem}
\begin{proof} 
    Let $r=r_1r_2$, where $r_1\geq 3$ and $r_2\geq 2$. Set $b=1$ if $k\neq 2\pmod {r_2}$ and set $b=0$ if $k= 2 \pmod {r_2}$. We consider the constraints $\psi_{\ZZ_{r_1},0}$ and $\psi_{\ZZ_{r_2},b}$, and let $\theta:\ZZ_{r_1}\to \Pi(\ZZ_{r_2})$ be such that $\theta_{(0,\dots, 0)}$ is the transposition of $0$ and $1$, and $\theta_\sigma=\Id$ for all $\sigma\in \ZZ_{r_1}^k\setminus\{(0,\dots, 0)\}$. By Lemma \ref{prodUEcom}, $\psi=\psi_{\ZZ_{r_1},0}\times_\theta \psi_{\ZZ_{r_2},1}$ is a $k$-ary commutative UE constraint on $\ZZ_{r_1}\times \ZZ_{r_2}$. We show that $\psi$ is not equivalent to any group constraint function. Define $f$ by
    \begin{align*}
        f(\sigma_1\ldots \sigma_{k-1})=f_\psi\bigl (f_\psi(\sigma_1)\ldots f_\psi(\sigma_{k-1})\bigr ),\quad \sigma_1,\ldots, \sigma_{k-1}\in \Omega^{k-1}
    \end{align*}
    Let $\eta\in (\ZZ_{r_1}\times \ZZ_{r_2})^{k-1}$ be such that $f_\psi(\eta)=(0,1)$. Let
    \begin{align*}
        \sigma &= ((-1,0),(1,0),(0,0),\ldots, (0,0))\in (\ZZ_{r_1}\times \ZZ_{r_2})^{k-1}\\
        \sigma' &= ((1,0),(1,0),(0,0),\ldots, (0,0))\in (\ZZ_{r_1}\times \ZZ_{r_2})^{k-1}\\
        \sigma'' &= ((-1,0),(-1,0),(0,0),\ldots, (0,0))\in (\ZZ_{r_1}\times \ZZ_{r_2})^{k-1}.
    \end{align*}
    Since $f\in F_2(\{\psi\})$, by Theorem \ref{symthm},  it suffices to prove that
    \begin{equation*}
        f(\sigma'\sigma''\eta^{k-3})\neq f(\sigma\sigma\eta^{k-3}).
    \end{equation*}
    Indeed,
    \begin{align*}
        f_\psi(\sigma) = (0,b),\quad f_\psi(\sigma')=(-2,b),\quad f_\psi(\sigma'')=(2,b),
    \end{align*}
and so we have
    \[
       f(\sigma'\sigma''\eta^{k-3})=f_\psi((-2,b),(2,b),(0,1),\ldots, (0,1))=(0,-b+3-k).
    \]
    On the other hand, since $\theta_{(0,\ldots, 0)}$ permutes $0$ and $1$,
    \begin{align*}
        f(\sigma\sigma\eta^{k-3})=f_\psi((0,b),(0,b),(0,1),\dots,(0,1))=\begin{cases}
            (0,0) & b=1\\
            (0,\theta_{(0,\ldots, 0)}^{-1}(-2)) & b=0.
        \end{cases}
    \end{align*}

If $b=1$ then $k\neq 2\pmod {r_2}$, so $-b+3-k=2-k\neq 0\pmod{r_2}$. If $b=0$ then $k=2\pmod{r_2}$ and moreover $r_2\ge 3$, as the case $r_2=2$ and $k$ being even is excluded by the hypotheses of the theorem. In this case we have $-b+3-k=3-k=1\pmod{r_2}$. If $r_2\geq 4$ then $\theta_{(0,\ldots, 0)}^{-1}(-2)=-2\neq 1\pmod{r_2}$. If $r_2=3$, then $\theta_{(0,\ldots, 0)}^{-1}(-2)=0\neq 1\pmod{r_2}$. Therefore $f(\sigma'\sigma''\eta^{k-3})\neq f(\sigma\sigma\eta^{k-3})$.
\end{proof}

\section*{Acknowledgement}
We thank Noela M\"{u}ller for bringing this problem to our attention and for the initial discussions.

 \bibliographystyle{abbrv}
 \bibliography{theo}

\section*{Appendix}\label{apdx:quazigroups}

Here we briefly sketch how the special cases of Theorems~\ref{symthm} and~\ref{cor:reducible} when $\Psi=\{\psi\}$, restated below, can be alternatively proved using quasigroup theory.  Recall that $f^{(i)}_\psi$ denotes the unique function in $F_i(\{\psi\})$.

\begin{theorem}\label{thm:specialcase}
        Let $\psi$ be a $k$-ary UE commutative constraint function with $k\geq 3$. Then $\psi$ is equivalent to a group constraint function if and only if $f^{(2)}_\psi$ is a symmetric function.
\end{theorem}
\begin{theorem}\label{thm:specialcasereducible}
        Let $\psi$ be a $k$-ary UE commutative constraint function with $k\geq 4$. Then $\psi$ is equivalent to a group constraint function if and only if $\psi$ is reducible.
\end{theorem}

\renewcommand{\thesection}{A}
\setcounter{theorem}{0}
\subsection*{A\ \  Symmetry and proof of Theorem~\ref{thm:specialcase}}
\begin{definition}
    Let $\Omega$ be a nonempty set and let $f:\Omega^k\to \Omega$. Then $(\Omega,f)$ is a $k$-quasigroup if, for all $d\in [k]$ and $\tau_0,\tau_1,\ldots,\tau_{d-1},\tau_{d+1},\ldots \tau_k\in \Omega$, there is a unique $\tau \in \Omega$ such that 
    \begin{equation*}
        f(\tau_1\ldots \tau_{d-1}\tau\tau_{d+1}\ldots \tau_k)=\tau_0
    \end{equation*}
\end{definition}
\begin{definition}
    Let $(\Omega,f)$ be a $k$-quasigroup and let $\alpha\in \Omega$.  Then $\alpha$ is neutral in $(\Omega,f)$ if, for all $\tau\in \Omega$ and $i\in [k]$
    \begin{equation*}
        f(\alpha^{i-1}\tau\alpha^{k-i})=\tau.
    \end{equation*}
\end{definition}
\begin{definition}
    Let $\psi$ be a $k$-ary UE constraint. Fix $\alpha\in \Omega$ and let $\beta=f_\psi(\alpha\ldots \alpha)$. Let $f'_{\psi,\alpha}:\Omega^k\to \Omega$ be given by
\begin{equation*}
    f'_{\psi,\alpha}(\tau_1\ldots \tau_k)=f_\psi(f_\psi(\tau_1\ldots \tau_{k-1})f_\psi(\tau_k\alpha^{k-3}\beta)f_\psi(\alpha^{k-2}\beta)\ldots f_\psi(\alpha^{k-2}\beta))
\end{equation*}
\end{definition}
By definition it is straightforward to prove the following.
\begin{lemma}
    Suppose $\psi$ is UE and let $\alpha\in \Omega$. Then $(\Omega,f_{\psi,\alpha}')$ is a $k$-quasigroup with neutral element $\alpha$.
\end{lemma} 
The next lemma characterises satisfiability of $\psi$ by quasigroup equations.

\begin{lemma}
    Suppose $\psi$ is $k$-ary, commutative, and UE. Suppose $\alpha\in \Omega$ and $\beta=f_\psi(\alpha\ldots\alpha)$. Then, for $\pi\in \Omega^k$, we have $f'_{\psi,\alpha}(\pi)=\beta$ if and only if $\pi\models \psi$.
\end{lemma}
\begin{proof}
    Suppose $\pi\in\Omega^k$ and write $\pi=\pi_1\ldots \pi_k$. We have
    \begin{align*}
        f'_{\psi,\alpha}(\pi)=\beta&\iff f_\psi(\pi_1\ldots \pi_{k-1})f_\psi(\pi_k\alpha\ldots\alpha\beta)f_\psi(\alpha\ldots \alpha\beta)\ldots f_\psi(\alpha\ldots\alpha\beta)\beta\models \psi\\
        &\iff f_\psi(\pi_1\ldots \pi_{k-1})f_\psi(\pi_k\alpha\ldots \alpha\beta)\alpha\ldots \alpha\beta\models\psi\\
        &\iff f_\psi(f_\psi(\pi_1\ldots \pi_{k-1})\alpha\ldots \alpha\beta)=f_\psi(\pi_k\alpha\ldots \alpha\beta)\\
        &\iff f_\psi(\pi_1\ldots \pi_{k-1})=\pi_k\\
        &\iff\pi_1\ldots \pi_k\models \psi.
    \end{align*}
\end{proof}

{\em Proof of Theorem~\ref{thm:specialcase}.\ }
We apply a theorem that was
cited in a paper~\cite{dudek2006quasigroups} by Wieslaw Dudek and Vladimir Mukhin, which was originally proved by D\"ornte in a German paper~\cite{Dörnte1929}.
\begin{theorem}[D\"ornte]
    Let $(\Omega,f)$ be an $n$-ary quasigroup containing a neutral element. Then there is a binary group $(\Omega,+)$ such that $f(x_1,\ldots, x_n)=x_1+\ldots +x_n$ for all $x_1,\ldots, x_n\in \Omega$ if and only if $f$ is associative.
\end{theorem}

With the help of D\"ornte's theorem, to complete the proof for Theorem~\ref{thm:specialcase} it suffices to prove the following, which follows by definition of $f'_{\psi,\alpha}$.
\begin{lemma}
    Suppose $\psi$ is $k$-ary, commutative, and UE. Then $f_{\psi,\alpha}'$ is associative for some $\alpha\in \Omega$ if and only if $f_\psi^{(2)}$ is symmetric.
\end{lemma}

\setcounter{theorem}{0}
\renewcommand{\thesection}{B}

\subsection*{B\ \ Reducibility and proof of Theorem~\ref{thm:specialcasereducible}}

We sketch an alternative proof for the special case of Theorem~\ref{cor:reducible} when $\Psi=\{\psi\}$ using ``reducibility'' of quasigroups. 
In quasigroup theory, the term reducibility refers to a notion that differs from the one we introduced in Definition~\ref{def:reducibility}. To avoid confusion, we will refrain from using the word reducibility by itself when we discuss quasigroups. Instead, we introduce the terms {\em exchangeability} and {\em group-reducibility}: Exchangeability for quasigroups, defined below, corresponds exactly to the notion defined as reducibility in Definition~\ref{def:reducibility} for UE constraint functions, whereas group-reducibility refers to the standard notion of reducibility in quasigroup theory.

\begin{definition}
    A $(k-1)$-quasigroup $(\Omega,f)$ is exchangeable if for all $1\leq i\leq j\leq k-1$, $\sigma_1,\sigma_1'\in\Omega^{i-1}$, $\sigma_2,\sigma_2'\in \Omega^{k-j-1}$, and $\eta,\eta'\in \Omega^{j-i+1}$ such that
    \begin{equation*}
        f(\sigma_1\eta\sigma_2)=f(\sigma_1\eta'\sigma_2)
    \end{equation*}
    we also have
    \begin{equation*}
        f(\sigma_1'\eta\sigma_2')=f(\sigma_1'\eta'\sigma_2').
    \end{equation*}
\end{definition}

\begin{definition}
    A $(k-1)$-quasigroup $(\Omega,f)$ is strongly group-reducible if there are binary quasigroups $(\Omega,g_1),\ldots,(\Omega,g_{k-2})$ such that
    \begin{equation*}
        f(x_1\ldots x_{k-1})=g_1(g_2(g_3(\ldots g_{k-3}(g_{k-2}(x_1,x_2),x_3)\ldots ))).
    \end{equation*}
\end{definition}

\begin{definition}
    A $(k-1)$-quasigroup $(\Omega,f)$ is totally symmetric if, for all $(x_1,\ldots,x_k)\in \Omega^k$ and all $\pi\in \Pi([k])$,
    \begin{equation*}
        f(x_1,\ldots, x_{k-1})=x_k\iff f(x_{\pi(1)},\ldots,x_{\pi(k-1)})=x_{\pi(k)}.
    \end{equation*}
\end{definition}

\noindent {\em Note}. This notion of total symmetry in quasigroup corresponds to the commutativity of UE constraints we defined in our paper.

\smallskip

There is a notion of ``completely reducible" quasigroups used in \cite{stojakovic1994quasigroups}. That notion is weaker in the sense that if a quasigroup is strongly group-reducible then it is completely reducible. We refer the reader to \cite{stojakovic1994quasigroups} for the precise definition of completely reducible quasigroups.

The following is a direct corollary of
\cite[Theorem 5]{belousovsandik1966quasigroups} by  taking $k_i=1$ for $i=1\dots k-2$ in their statement.
\begin{theorem}[Belousov-Sandik \cite{belousovsandik1966quasigroups}, 1966]\label{thm:belousovred}
    If $(\Omega,f)$ is an exchangeable $(k-1)$-quasigroup, then $(\Omega,f)$ is strongly group-reducible.
\end{theorem}
We also apply the following theorem due to Stojakovi\'c.
\begin{theorem}[Stojakovi\'c \cite{stojakovic1994quasigroups}, 1994]\label{thm:stojakovicgroupred}
    For every completely group-reducible totally symmetric $(k-1)$-quasigroup $(\Omega,f)$ with $k-1\geq 3$, there is an Abelian group $(\Omega,+)$, a permutation $\fhi$ of $\Omega$, and $b\in \Omega$ such that, for all $\sigma=\sigma_1\ldots \sigma_{k-1}\in \Omega^{k-1}$,
    \begin{equation*}
        \fhi(f(\sigma))+\sum_{i=1}^{k-1}\fhi(\sigma_i)=b.
    \end{equation*}
\end{theorem}

{\em Proof of Theorem~\ref{thm:specialcasereducible}.\ } 
First suppose that there is a group $(\Omega,+)$ and $b\in \Omega$ such that $\psi=\psi_{(\Omega,+),b}$. To show that $\psi$ is reducible, let $d\in \{0,\ldots, k\}$ and let $\sigma,\sigma\in\Omega^d$, $\eta,\eta'\in \Omega^{k-d}$ be such that
\begin{equation}
    \sigma\eta\models \psi,\quad \sigma'\eta\models \psi,\quad \sigma\eta'\models \psi'.\label{eq:redhyp}
\end{equation}
Since $\psi=\psi_{(\Omega,+),b}$,~\eqref{eq:redhyp} is equivalent to
\begin{equation*}
    \sigma^\oplus+\eta^\oplus=b,\quad \sigma'^\oplus+\eta^\oplus=b,\quad \sigma^\oplus+\eta'^\oplus=b
\end{equation*}
which implies $\sigma^\oplus=\sigma'^\oplus$ and $\eta^\oplus =\eta'^\oplus$. Thus $\sigma'^\oplus +\eta'^\oplus=\sigma^\oplus +\eta^\oplus=b$, and therefore $\sigma'\eta'\models \psi$. Therefore $\psi$ is reducible.

For the converse direction, suppose $\{\psi\}$ is reducible. By Theorem \ref{thm:stojakovicgroupred}, it suffices to show that the $(k-1)$-quasigroup $(\Omega,f_\psi)$ is totally symmetric and completely group-reducible. To show that $(\Omega, f_\psi)$ is totally symmetric, let $\sigma_1,\ldots,\sigma_{k}\in\Omega^{k}$, and let $\pi\in \Pi([k])$. Suppose $f_\psi(\sigma_1,\ldots, \sigma_{k-1})=\sigma_k$. Then we have
\begin{equation*}
    \sigma_1,\ldots, \sigma_{k}\models\psi
\end{equation*}
and since $\psi$ is commutative,
\begin{equation*}
    \sigma_{\pi(1)},\ldots , \sigma_{\pi(k)}\models \psi.
\end{equation*}
Hence $f_\psi(\sigma_{\pi(1)},\ldots , \sigma_{\pi(k-1)})=\sigma_{\pi(k)}$, and hence $(\Omega,f_\psi)$ is totally symmetric.

To show that $(\Omega,f_\psi)$ is completely group-reducible, by Theorem~\ref{thm:belousovred} it suffices to show that $(\Omega,f_\psi)$ is exchangeable. Let $1\leq i\leq j\leq k-1$, $\sigma_1,\sigma_1'\in \Omega^{i-1}$, $\sigma_2,\sigma_2'\in \Omega^{k-j-1}$, and $\eta,\eta'\in \Omega^{j-i+1}$ be such that
\begin{equation*}
    f_\psi(\sigma_1\eta\sigma_2)=f_\psi(\sigma_1\eta'\sigma_2)
\end{equation*}
By definition of $f_\psi$,
\begin{equation*}
    \sigma_1\eta\sigma_2f_\psi(\sigma_1\eta\sigma_2)\models \psi,\quad \sigma_1\eta'\sigma_2f_\psi(\sigma_1\eta'\sigma_2)\models \psi, \quad \sigma_1'\eta\sigma_2'f_\psi(\sigma_1'\eta\sigma_2')\models \psi
\end{equation*}
Since $\psi$ is commutative and $f_\psi(\sigma_1\eta\sigma_2)=f_\psi(\sigma_1\eta'\sigma_2)$, we get
\begin{equation*}
    \eta\sigma_1\sigma_2f_\psi(\sigma_1\eta\sigma_2)\models \psi,\quad \eta'\sigma_1\sigma_2f_\psi(\sigma_1\eta\sigma_2)\models \psi, \quad \eta\sigma_1'\sigma_2'f_\psi(\sigma_1'\eta\sigma_2')\models \psi.
\end{equation*}
Therefore, by reducibility of $\{\psi\}$ and Remark~\ref{r:extension-exchange}, we can replace $\eta$ in $\eta\sigma_1'\sigma_2'f_\psi(\sigma_1'\eta\sigma_2')\models \psi$ with $\eta'$ to get $\eta'\sigma_1'\sigma_2'f_\psi(\sigma_1'\eta\sigma_2')\models \psi$. Applying commutativity again gives
\begin{equation*}
  \sigma_1'\eta'\sigma_2'f_\psi(\sigma_1'\eta\sigma_2')\models \psi  
\end{equation*}
and thus $f_\psi(\sigma_1'\eta'\sigma_2')=f_\psi(\sigma_1'\eta\sigma_2')$. Therefore $(\Omega,f_\psi)$ is exchangeable, as desired.
\qed

\end{document}